\documentclass[10pt]{amsart}
\usepackage[english]{babel}
\usepackage{amssymb}
\usepackage{enumitem}
\usepackage{comment}
\usepackage{color}
\usepackage{hyperref}
\usepackage[initials]{amsrefs}
\usepackage[all]{xy}
\usepackage{tikz}
\usepackage{tikz-cd}
\usetikzlibrary{decorations.pathmorphing}

\usepackage{color}

\newcommand{\bbN}{{\mathbb N}}
\newcommand{\bbQ}{{\mathbb Q}}
\newcommand{\bbR}{{\mathbb R}}
\newcommand{\bbZ}{{\mathbb Z}}
\newcommand{\bbC}{{\mathbb C}}
\newcommand{\bbH}{\mathbb{H}}
\newcommand{\bbO}{\mathbb{O}}

\newcommand{\HYP}{\mathcal{H}}

\newcommand{\calC}{\mathcal{C}}

\newcommand{\calV}{\mathcal{V}}

\newcommand{\calN}{\mathcal{N}}
\newcommand{\calZ}{\mathcal{Z}}
\newcommand{\calO}{\mathcal{O}}

\newcommand{\bs}{\backslash}
\newcommand{\betti}{\beta^{(2)}}

\newcommand{\id}{\operatorname{id}}
\newcommand{\im}{\operatorname{im}}

\newcommand{\isom}{\operatorname{Isom}}

\newcommand{\pr}{\operatorname{pr}}

\newcommand{\vol}{\operatorname{vol}}

\newcommand{\Isom}{\operatorname{Isom}}

\newcommand{\Ker}{\operatorname{Ker}}

\newcommand{\Homeo}{\operatorname{Homeo}}
\newcommand{\Stab}{\operatorname{Stab}}
\newcommand{\SL}{\operatorname{SL}}

\newcommand{\Aut}{\operatorname{Aut}}
\newcommand{\Auttop}{\operatorname{Aut}^{c}}

\newcommand{\Inn}{\operatorname{Inn}}
\newcommand{\Out}{\operatorname{Out}}
\newcommand{\Outtop}{\operatorname{Out}^c}

\newcommand{\Radam}{\operatorname{Rad}_{\mathrm{am}}}

\newcommand{\PSL}{\operatorname{PSL}}
\newcommand{\PGL}{\operatorname{PGL}}
\newcommand{\GL}{\operatorname{GL}}

\newcommand{\QI}{\operatorname{QI}}

\newcommand{\fin}{{\mathrm{fin}}}

\newcommand{\acts}{\curvearrowright}

\newcommand{\CAF}{\ensuremath{\mathrm{(CAF)}}}
\newcommand{\NBC}{\ensuremath{\mathrm{(NbC)}}}
\newcommand{\wNBC}{\ensuremath{\mathrm{(wNbC)}}}

\newcommand{\IRR}{\ensuremath{\mathrm{(Irr)}}}
\newcommand{\BT}{\ensuremath{\mathrm{(BT)}}}
\newcommand{\NBCone}{\ensuremath{\mathrm{(wNbC)_1}}}
\newcommand{\Dreg}{\mathcal{D}_{\mathrm{reg}}}

\newcommand{\normal}{\triangleleft}

\newcommand{\overto}[1]{\,{\buildrel{#1}\over\longrightarrow}\,}

\newcommand{\ssf}[1]{{#1}^{\mathrm{ss}}}
\newcommand{\tdf}[1]{{#1}^{\mathrm{td}}}

\newtheorem{mthm}{Theorem}

\newtheorem{theorem}{Theorem}[section]
\newtheorem{lemma}[theorem]{Lemma}

\newtheorem{corollary}[theorem]{Corollary}
\newtheorem{cor}[theorem]{Corollary}
\newtheorem{proposition}[theorem]{Proposition}
\newtheorem{prop}[theorem]{Proposition}

\theoremstyle{definition}
\newtheorem{definition}[theorem]{Definition}
\newtheorem{defn}[theorem]{Definition}

\newtheorem{example}[theorem]{Example}
\newtheorem{prob}[theorem]{Problem}

\newtheorem{remark}[theorem]{Remark}

\newtheorem{setup}[theorem]{Setup}

\numberwithin{equation}{section}

\begin{document}

\title{Lattice envelopes}

\author{Uri Bader}
\address{Weizmann Institute, Rehovot}
\email{uri.bader@gmail.com}

\author{Alex Furman}
\address{University of Illinois at Chicago, Chicago}
\email{furman@uic.edu}

\author{Roman Sauer}
\address{Karlsruhe Institute of Technology}
\email{roman.sauer@kit.edu}

\thanks{In part, U.B. and A.F. were supported by BSF grant 2008267 and A.F. by NSF grant DMS 1611765. 
All three authors thank MSRI for its support during the \emph{Geometric Group Theory} Program.}


\begin{abstract}
We introduce a class of countable groups by some abstract group-theoretic conditions. It includes 
linear groups with finite amenable radical and finitely generated 
residually finite groups with some non-vanishing $\ell^2$-Betti numbers that are not virtually a product of two infinite groups. Further, 
it includes acylindrically hyperbolic groups. 
For any group $\Gamma$ in this class we determine the general structure of the possible lattice embeddings of $\Gamma$, i.e.~of all compactly generated, locally compact groups that contain $\Gamma$ as a lattice. 
This leads to a precise description of possible non-uniform lattice embeddings of groups in this class. 
Further applications include the determination of possible lattice embeddings of fundamental groups of closed manifolds with pinched negative curvature. 
\end{abstract}

\maketitle


\section{Introduction} 
\label{sec:introduction_and_statement_of_the_main_result}

\subsection{Motivation and background}
Let $G$ be a locally compact second countable group, hereafter denoted \textbf{lcsc}\footnote{If $G$ is, in addition, totally disconnected we call $G$ a \textbf{tdlc} group.}. 
Such a group carries a left-invariant Radon measure unique up to scalar multiple, 
known as the Haar measure. 
A subgroup $\Gamma<G$ is a \textbf{lattice} if it is discrete, and $G/\Gamma$ carries
a finite $G$-invariant measure; equivalently, if the $\Gamma$-action on $G$ admits
a Borel fundamental domain of finite Haar measure.
If $G/\Gamma$ is compact, one says that $\Gamma$ is a \textbf{uniform lattice}, 
otherwise $\Gamma$ is a \textbf{non-uniform} lattice. The inclusion $\Gamma\hookrightarrow G$ is called 
a \textbf{lattice embedding}. 
We shall also say that $G$ is a \textbf{lattice envelope} of $\Gamma$.
\medskip

The classical examples of lattices come from geometry and arithmetic.
Starting from a locally symmetric manifold~$M$ of finite volume, 
we obtain a lattice embedding $\pi_1(M)\hookrightarrow \Isom(\tilde{M})$ of its fundamental group 
into the isometry group of its universal covering via the action by deck transformations. 
The real Lie group $\isom(\tilde M)$ is semi-simple if $\tilde M$ has no Euclidean direct 
factors, and the lattice $\Gamma<\isom(\tilde M)$ is uniform if and only if $M$ is compact. 

Real semi-simple Lie groups viewed as algebraic groups admit \textbf{arithmetic lattices} such as 
$\SL_d(\bbZ)<\SL_d(\bbR)$~\cite{Margulis:book}*{Theorem 3.2.4~on p.63}. 
Analogous constructions exist for products of real and $p$-adic algebraic semi-simple 
groups, such as 
\[
	\SL_d(\bbZ[\frac{1}{p_1},\dots,\frac{1}{p_k}])<
	\SL_d(\bbR)\times \SL_d(\bbQ_{p_1})\times\cdots\times\SL_d(\bbQ_{p_k}).
\]
Our notations and conventions for arithmetic lattices are defined in \S\ref{sub:arithmetic} below.

\medskip

A central theme in the study of lattices are the connections between the lattices
and the ambient group.
Mostow's Strong Rigidity (\cite{Mostow:1973book} for the uniform case, \cites{Prasad-Mostow, Margulis-Mostow} for the non-uniform case) 
asserts that an irreducible lattice in a semi-simple real Lie group $G$ that is not locally isomorphic to $\SL_2(\bbR)$
determines the ambient Lie group among all Lie groups and
determines its embedding in the ambient Lie group uniquely up to an automorphism of~$G$.

It is natural to ask to what extent lattices determine their lattice envelopes among all lcsc groups, 
and which countable groups have only trivial lattice envelopes. 
To make the question precise we introduce the notion of \textbf{virtual isomorphism} between 
lattice  embeddings in Definition~\ref{def: virtual isomorphism}. 
A lattice embedding is called \textbf{trivial} if it is virtually isomorphic to the identity homomorphism of a countable discrete group, 
which can be regarded as a lattice embedding. Virtual isomorphism is an equivalence relation. We refer to \S\ref{sub: virtual iso} for more information. 

\begin{prob}\label{pro:envelopes} Given a countable group $\Gamma$, describe all its possible lattice envelopes up to virtual isomorphism. 
\end{prob}

In the study of lattices, the non-uniform lattices are often harder to work with when compared to uniform ones.
For example, the condition of \textbf{integrability} of lattices 
(e.g.~required for different purposes in 
\citelist{\cite{Margulis-T}\cite{Shalom}\cite{Monod}\cite{BFGM}\cite{GKM}}) 
is automatically satisfied by uniform lattices and can be proven for some examples of 
non-uniform lattices, often using elaborate arguments.
This motivates the following.
\begin{prob}\label{pro:non-uni}
	How prevalent are non-uniform lattice embeddings? 
	What groups admit both non-trivial uniform and non-uniform embeddings?
\end{prob}

\subsection{Main result}
To state our main result, Theorem~\ref{T:main}, towards Problem~\ref{pro:envelopes} we need to put some conditions on the group~$\Gamma$. 

\begin{definition}\label{D:properties}
	We say that a countable group $\Gamma$ has property
	\begin{itemize}
		\item \textbf{(BT)} 
		if there is an upper bound on the order of finite subgroups of $\Gamma$.
		\item \textbf{(Irr)}   
		if $\Gamma$ is not virtually isomorphic to a product
		of two infinite groups.
		\item \textbf{(CAF)} if every amenable commensurated subgroup of $\Gamma$ is finite.
		\item \textbf{(wNbC)} if for every normal subgroup $N\normal \Gamma$ and every 
		commensurated subgroup $M<\Gamma$ such that $N\cap M=\{1\}$ there is a finite index subgroup of $M$ that commutes with $N$.
		\item \textbf{(NbC)} if every quotient of $\Gamma$ by a finite normal subgroup has (wNbC). 
	\end{itemize}
\end{definition}
The abbreviations stand for \textbf{B}ound on \textbf{T}orsion, \textbf{Irr}educiblity,
\textbf{C}ommensurated \textbf{A}menable is \textbf{F}inite, and (\textbf{w}eak) \textbf{N}ormal \textbf{b}y \textbf{C}ommensurated.
For the definition of commensurated subgroup see \S\ref{sub: commensurated and commensurable group}. 

In \S\ref{sec:properties_bt_irr_caf_and_nbc} we prove the following result that shows that large classes of groups satisfy the above conditions, and we also describe these classes in more detail. 
\begin{theorem}\label{T:good-groups}
    The following countable groups have property $\CAF$: 
	\begin{enumerate}
		\item Groups with a non-vanishing $\ell^2$-Betti number in some degree. 
		\item Linear groups with finite amenable radical. 
		\item Groups in the class $\Dreg$ introduced in~\cite{thom}. 
		This class contains all acylindrically hyperbolic groups and thus non-elementary convergence groups. 
	\end{enumerate}	
    Further, 
	the following countable groups have property $\NBC$: 
	\begin{enumerate}
		\item[(1')] Groups with non-vanishing first $\ell^2$-Betti number.
		\item[(2')] Linear groups. 
		\item[(3')] Finitely generated residually finite groups with finite amenable radical. 
		\item[(4')] Groups in the class $\Dreg$. 
	\end{enumerate}
	 The groups in~(1') and (4') also have the property~$\IRR$. 
\end{theorem}

The statement about $\NBC$~for the groups of class (3') is due to 
Caprace--Kropholler--Reid--Wesolek~\cite{capraceetal-closure}*{Corollary 18} and the fact 
that the class (3') is closed under passing to quotients by finite normal subgroups. 

The next theorem was announced in~\cite{BFS-CRAS} with a slightly stronger formulation of the \NBC~and \CAF~condition. It gives a partial answer to Problem~\ref{pro:envelopes} that concerns 
possible lattice envelopes of a given group $\Gamma$.  
We impose no restriction on the lattice envelope, except for it being compactly generated. 
This condition can also be removed if $\Gamma$ is assumed to be finitely generated (see Lemma~\ref{L:fg-cg}).

\begin{mthm}[Structure of possible lattice envelopes]\label{T:main}\hfill{}\\
	Let $\Gamma$ be a countable group with properties 
	$\IRR$, $\CAF$, and $\NBC$, and  
	let $\Gamma< G$ be a non-trivial lattice embedding into a compactly generated 
	lcsc group $G$.  
	
    Then the lattice embedding $\Gamma\hookrightarrow G$ is virtually isomorphic to one of the following lattice embeddings: 
	\begin{enumerate}
		\item
		\label{case:real-lat}
		an irreducible lattice in a connected, center-free,
		semi-simple real Lie group without compact factors; 
		\item
		\label{case: arith-lat}
		an $S$-arithmetic lattice or an $S$-arithmetic lattice up to tree extension as in 
		Definition~\ref{def: definition of S-arithmetic lattice up to tree extension};
		\item 
		\label{case: tdcc-lat}
		a lattice in a non-discrete totally disconnected locally compact group with
		trivial amenable radical.
	\end{enumerate}
    If in addition $\Gamma$ has \BT, then the lattice in (3) is uniform. 
	%
\end{mthm}

For the notion of tree extension see Definition~\ref{defn: generalized S-arithmetic lattice embedding}. 
The three cases are distinct (see Remark~\ref{rem: meaning of S-arithmetic}). 
We show in Example~\ref{exa: exotic lattice} the necessity of the condition $\NBC$ for the theorem above. 
 
The classical case of lattices in semi-simple real Lie groups includes both arithmetic and 
non-arithmetic examples (non-arithmetic lattices 
can appear only in $\operatorname{SO}(n,1)$ and $\operatorname{SU}(n,1)$).
The case of $S$-arithmetic lattices in the above statement refers to irreducible
lattices in a product of finitely many real semi-simple 
Lie groups and $p$-adic ones, where factors of both types are present.
Case (\ref{case: tdcc-lat}) contains a large class of examples that includes lattices in $p$-adic groups,
and fundamental groups $\Gamma$ of some finite cube complexes whose universal cover
has a non-discrete group of automorphisms. 
This last case remains quite mysterious although a structure theory of non-discrete, 
simple, totally disconnected groups emerged in the last decade (in this regard, Caprace's survey~\cite{caprace-ecm} is recommended). 

However assuming property $\BT$ the last case does not
allow non-uniform lattices. 
Hence we obtain the following partial solution to Problem~\ref{pro:non-uni}. 


\begin{corollary}[Classification of non-uniform lattices]\hfill{}\\
	Let $\Gamma$ be a countable group with properties $\BT$, $\IRR$, $\CAF$, and $\NBC$. 	
	Then every non-uniform lattice embedding $\Gamma\hookrightarrow G$ of $\Gamma$ 
	is virtually isomorphic to either 
	\begin{enumerate}
		\item
		an irreducible lattice in a connected, center free,
		semi-simple real Lie group without compact factors, or 
		\item
		an $S$-arithmetic lattice, possibly up to tree extension, 
		with both real and non-Archimedean factors present. 
	\end{enumerate}
	In particular, every lattice embedding of such $\Gamma$ is square integrable, 
	with the exception of lattice embeddings into Lie groups that are locally isomorphic to $\SL_2(\bbR)$ or $\SL_2(\bbC)$. 
\end{corollary}

This is a direct corollary except for the last statement which is a consequence of results about 
square integrability of classical lattices from Shalom's work (see~\cite{Shalom}*{\S2} 
for the S-arithmetic and higher rank case and~\cite{shalom-square}*{Theorem~3.7} for the rank~$1$ case.

\subsection{Applications}
As applications of our main result, we now present more precise classification results of lattice envelopes for specific groups. 
Their proofs can be found in~\S\ref{sec: applications}. 

\begin{mthm}[Mostow rigidity with locally compact targets]\label{T:Mostow}\hfill{}\\
	Let $\Gamma<H$ be an irreducible lattice embedding into a center-free, real semi-simple Lie group~$H$ 
	without compact factors that is not locally isomorphic to $\SL_2(\bbR)$.
	Then every non-trivial lattice embedding of $\Gamma$ into an lcsc group~$G$ is virtually isomorphic to $\Gamma<H$. 
	
	If $\Gamma<H$ is an $S$-arithmetic lattice embedding as in Definition~\ref{def: definition of S-arithmetic lattice up to tree extension}.  
	Then every non-trivial lattice embedding of $\Gamma$ into an lcsc group~$G$ is virtually isomorphic to $\Gamma<H$ or
	to a tree extension $\Gamma<H^*$ if $H$ includes rank~$1$ non-archimedean factors. 
\end{mthm}
This theorem may be viewed as a generalization of Mostow's Strong Rigidity,
corresponding to the special case of Theorem~\ref{T:Mostow}, 
where $H$ is a real semi-simple Lie group and $G$ is also known to be a semi-simple real Lie group
(in most accounts one even assumes $G=H$ and the focus is on aligning two given lattice embeddings of $\Gamma$ in $G$ by
an automorphism of $G$).
In fact, Mostow considered the case of uniform lattice embeddings \cite{Mostow:1973book};
the non-uniform were later obtained by Prasad \cite{Prasad-Mostow} and Margulis \cite{Margulis-Mostow}. 
In \cite{Furman-MM} the second author proved Theorem~\ref{T:Mostow}
for the case of a simple real Lie group $H$ with ${\rm rk}_\bbR(H)\ge 2$ and a general lattice envelope $\Gamma<G$, 
and for the case of rank-one real Lie group $H$ and a uniform lattice envelope $\Gamma<G$.

The case of $H=\PSL_2(\bbR)$ is excluded from Theorem~\ref{T:Mostow} 
for two reasons. First, strong rigidity does not hold in this setting -- the moduli space
of embeddings of a surface group in the same group $G=H=\PSL_2(\bbR)$ is multidimensional.
Secondly, non-uniform lattices in $\PSL_2(\bbR)$, that are virtually free groups of finite 
rank $2\le n<\infty$, 
can be embedded as a lattice in completely different lcsc groups, such as $\PGL_2(\bbQ_p)$, 
or the automorphism group $\Aut(T)$ of a locally finite tree $T$. 
Our next results show that all lattice envelopes of free groups are related to the ones mentioned above. 

\begin{mthm}[Lattice embeddings of free groups]\label{T:free-groups} \hfill{}\\
	Let $\Gamma$ be a finite extension of a finitely generated non-abelian 
	free group $F_n$ and let $\Gamma<G$ be a non-trivial lattice embedding. 
	\begin{enumerate}
	\item
		If $\Gamma<G$ is non-uniform, then it is virtually isomorphic to a non-uniform lattice embedding into $\PSL_2(\bbR)$. 
	\item
		If $\Gamma<G$ is uniform, then it is virtually isomorphic to a lattice embedding 
		into a closed cocompact subgroup of the automorphism group of a tree. 
	\end{enumerate}
\end{mthm}
Note that the second possibility in the Theorem includes such examples as $\PGL_2(\bbQ_p)$
or other rank-one non-Archimedean groups. 

For uniform lattices in $\PSL_2(\bbR)$, such as surface groups, 
the possibilities for lattice embeddings are even more restricted.

\begin{mthm}[Lattice embeddings of surface groups]\label{T:surface-groups}\hfill{}\\
	Let $\Gamma$ be a uniform lattice in $\PSL_2(\bbR)$. Any other non-trivial lattice embedding 
	of $\Gamma$ is virtually isomorphic to a uniform lattice embedding into $\PSL_2(\bbR)$. 
\end{mthm}

Let $M$ be a manifold that admits a Riemannian metric with strictly negative sectional curvature, but is not homeomorphic to a locally symmetric one.
We conjecture that fundamental groups of such manifolds $\Gamma=\pi_1(M)$ have no non-trivial lattice envelopes. Under an additional pinching assumption, we are able to prove this. 

\begin{mthm}[Lack of lattice embeddings in pinched negative curvature]\label{T:GromovThurston}\hfill{}\\
	The fundamental group of a closed Riemannian manifold~$M$ of dimension $n\ge 5$ whose sectional curvatures range in 
	\[ \bigl[ -(1+\frac{1}{n-1})^2, -1\bigr]\]
	does not admit a non-trivial lattice embedding unless $M$ is homeomorphic to a closed hyperbolic manifold. 
\end{mthm}

Some remarks about this result are in order. 
Gromov--Thurston~\cite{Gromov+Thurston} construct infinitely many examples of negatively curved manifolds with the above pinching in 
dimensions $n\ge 4$ that are not homeomorphic to hyperbolic manifolds. By taking 
connected sums of a hyperbolic manifold with an exotic sphere Farrell--Jones construct closed 
smooth manifolds that are homeomorphic to a hyperbolic manifold but whose smooth structure does not support a hyperbolic Riemannian metric~\cite{farrell+jones-exotic}. We recommend the survey~\cite{farrell+jones+ontaneda} by Farrell--Jones--Ontaneda on these issues.  
 
Finally, we obtain the following surprising characterization of free groups. 

\begin{mthm}[Non-uniform lattice embeddings of groups with $\beta^{(2)}_1>0$]\label{T: characterisation of groups with positive first L2 Betti} \hfill{}\\
	Let $\Gamma$ be a group with non-vanishing first $\ell^2$-Betti number that has 
	an upper bound on the order of its finite subgroups.
	If $\Gamma$ possesses a non-uniform 
	compactly generated lattice envelope, then $\Gamma$ has a non-abelian free subgroup 
	of finite index. 
\end{mthm}

\subsection{Structure of the paper}

We devote most of~\S\ref{sec:properties_bt_irr_caf_and_nbc} to the proof of Theorem~\ref{T:good-groups}. 
We discuss an example by Burger--Mozes in~\S\ref{sub: examples_of_groups_that_are_not_in_bbc_} that 
shows the necessity of the condition~\NBC; if one drops \NBC~while still keeping \CAF~then there are 
exotic lattice embeddings that are not covered by the three cases of Theorem~\ref{T:main}. 
\smallskip

The definition of virtually isomorphism of lattice embeddings and the tools of the trade for the 
proof of Theorem~\ref{T:main} are provided in~\S\ref{sec:preliminaries}. 
The most difficult proof in~\S\ref{sec:preliminaries} is the result about outer automorphisms of S-arithmetic lattices (Theorem~\ref{thm: outer automorphism groups are finite}). 
While this result is known or expected to hold by the experts there was no proof so far in the required generality. 
We hope its proof provides a useful reference. 
\smallskip

The bulk of the paper is devoted to the proof of Theorem~\ref{T:main} in~\S\ref{sec:proof of main result}. 
The paper~\cite{BFS-CRAS} in which we announced the results of this paper and provided proofs in special cases 
might be helpful for the reader. 
In the first step of the proof we rely on property \CAF~and the positive solution of Hilbert's 5th problem to show that the given lattice embedding $\Gamma<G$ is virtually isomorphic to one into a product of a semi-simple Lie group~$L$ 
and a totally disconnected group~$D$. 
So we may assume that $G=L\times D$. Let $U<D$ be a compact-open subgroup. In the second step of the proof 
we split $L$ further as $L=L_0\times L_1$ in such a way that the projection of $\Gamma$ to $L_1$ is dense. 
Depending on the finiteness of $\Gamma\cap (L\times U)$ and $\Gamma\cap (\{1\}\times L_1\times U)$, 
we distinguish three cases in the third step which correspond to the cases (1), (2) or (3) of the statement of Theorem~\ref{T:main}. 
The second case is the most sophisticated. 
Here we have to 
identify the lattice $\Gamma\cap(\{1\}\times L_1\times D')<L_1\times D'$, where $D'<D$ is a certain closed cocompact subgroup, 
as an S-arithmetic lattice, and $D'$ with the corresponding non-Archimedean factor. This identification is an arithmeticity theorem which is proved in our companion paper~\cite{BFS-adelic} and heavily relies on Margulis' commensurator rigidity. We actually prove in that paper a result of greater generality which does not assume compact generation of the ambient group. The version needed here is only slightly more general than the arithmeticity theorem by Caprace--Monod in~\cite{Caprace+Monod:II}*{Theorem~5.20}. 
Our basic setup for $S$-arithmetic groups and a detailed explanation of the arithmeticity theorem is provided in~\S\ref{sec:arithmetic_case}. Further, we use Margulis' normal subgroup theorem and the aforementioned result on outer automorphism groups to conclude that $L=L_1$ and $L_0=\{1\}$ in the second case. 
The condition on \NBC~appears just once in the whole proof, namely in the third step of the proof, in Lemma~\ref{lem: normalizing}.  

The final step of the proof identifies the difference between $L\times D'$ and $L\times D$ as a tree extension in the sense of Setup~\ref{setup:arith}. 
\smallskip

In the proof of the applications, that is 
Theorems~\ref{T:Mostow}--\ref{T: characterisation of groups with positive first L2 Betti}, 
a major step is always to identify the most mysterious case, namely (3), in Theorem~\ref{T:main}. 
To this end, we appeal to ideas of geometric group theory and quasi-isometric rigidity. 
This is based on the fact that a lattice with property \BT~is quasi-isometric to any of its totally disconnected lattice envelopes.

\section{The properties  $\BT$, $\IRR$, $\CAF$, and $\NBC$} 
\label{sec:properties_bt_irr_caf_and_nbc}

We describe the classes in Theorem~\ref{T:good-groups} in more detail. 
In~\S\ref{sub: proof of properties of groups} we prove Theorem~\ref{T:good-groups}. 
In~\S\ref{sub: examples_of_groups_that_are_not_in_bbc_} we describe an exotic lattice embedding 
of a group that does not have property \NBC. 

\subsection{Commensurated and commensurable groups} 
\label{sub: commensurated and commensurable group}

\begin{definition}\label{def: commensurated}
	A subgroup $A$ of a group $G$ is \emph{commensurated} by a subset $S\subset G$ 
if $s^{-1}As\cap A$ has finite index in $s^{-1}As$ and $A$ for every $s\in S$. 
If $A<G$ is commensurated by all of $G$ we just say that $A<G$ is \emph{commensurated}.
\end{definition}

\begin{definition}\label{def: commensurable groups}
Two lcsc groups $G$ and $G'$ are \emph{commensurable} if there are open subgroups $H<G$ and 
$H'<G'$ of finite index such that $H\cong H'$ as topological groups. 

Further, $G$ and $G'$ are \emph{weakly commensurable} or \emph{virtually isomorphic} if there 
are open subgroups $H<G$ and $H'<G'$ of finite index and compact normal subgroups 
$K\normal H$ and $K'\normal H'$ such that $H/K\cong H'/K'$ as topological groups. 

A topological isomorphism $f\colon H/K\to H'/K'$ is called a \emph{virtual isomorphism} from $G$ to  $G'$. We say it \emph{restricts} to subgroups $A<G$ and $A'<G'$ if $f$ restricts to an isomorphism 
$(A\cap H)/(A\cap K)\to (A'\cap H')/(A'\cap K')$, which is then a virtual isomorphism between $A$ and $A'$.  
\end{definition}

If in the above definition the groups $G$ and $G'$ are countable discrete then a finite index subgroup of $G$ or $G'$  is open and a compact normal subgroup of $G$ or $G'$ is finite. 

The proof of the following easy lemma is left to the reader. 

\begin{lemma}\label{lem: permanence commensurated}Being commensurated is preserved in the following situations: 
\begin{enumerate}
\item Preimages of commensurated subgroups under homomorphisms are commensurated. 
\item Intersections of finitely many commensurated subgroups are commensurated.
\item A finite index subgroup of a commensurated subgroup is itself commensurated. 
\end{enumerate}
\end{lemma}

Both \CAF~and \NBC~are not necessarily preserved if one passes to a weakly commensurable group. They are, however, preserved by passing to quotients by finite kernels. 

\begin{lemma}\label{lem: CAF and amenable kernel}
	Let $K\normal \Gamma$ be a finite normal subgroup. If $\Gamma$ has property 
	\CAF, then $\Gamma/K$ has \CAF. 
\end{lemma}

\begin{proof}
	Let $p\colon \Gamma\to\Gamma/K$ be the projection. Let $A<\Gamma/K$ be an amenable, commensurated subgroup.  
	Then $p^{-1}(A)<\Gamma$  
	is commensurated by Lemma~\ref{lem: permanence commensurated}. 
	Since $p^{-1}(A)$ is a finite extension of $A$, it is amenable too. 
	By the property $\CAF$ the group $p^{-1}(A)$ is finite, thus $A$ is finite. 
\end{proof}

\begin{lemma}\label{lem: NBC and finite kernel}
	Let $K\normal\Gamma$ be a finite normal subgroup. If $\Gamma$ has property \NBC, 
	then $\Gamma/K$ has \NBC. 
\end{lemma}

\begin{proof}
The stated property is obvious and built into the definition of \NBC. 
\end{proof}

\subsection{Groups in the class $\Dreg$ and acylindrically hyperbolic groups}

The class $\Dreg$ was introduced by Thom~\cite{thom}. It is closely related to the class $\calC_{\mathrm{reg}}$ of Monod--Shalom. Both $\Dreg$ and $\calC_{\mathrm{reg}}$ are an attempt to define negative curvature for groups in a cohomological way. The definition of $\Dreg$ is analytical. 

Let $\pi\colon\Gamma\to U(\mathcal{H})$ be a unitary representation. 
A map $c\colon \Gamma\to \mathcal{H}$ is a \emph{quasi-1-cocycle} for $\pi$ if 
\[ 
	\Gamma\times\Gamma\ni (g,h)\mapsto \pi(g)c(h)-c(gh)+c(g)\in \mathcal{H}
\]
is uniformly bounded on $\Gamma\times\Gamma$. 
The vector space of quasi-1-cocycles modulo the bounded 
ones forms a group $QH^1(\Gamma, \mathcal{H})$. 
The class $\Dreg$ is the class of groups for which $QH^1(\Gamma, \ell^2(\Gamma))$ has non-zero $L(\Gamma)$-dimension in the 
sense of L\"uck~\cite{lueck-book}*{Section~6.1}. More concrete is the following description: 

\begin{proposition}[\cite{thom}*{Corollary~2.5 and Lemma~2.8}]\label{prop: nonamenable groups in Dreg}
	A group~$\Gamma$ is in $\Dreg$ if and only if there is an unbounded 
quasi-1-cocycle $\Gamma\to\ell^2(\Gamma)$. 
\end{proposition}

The following theorem was proved by Bestvina--Fujiwara~\cite{bestvina+fujiwara} and Hamenst\"adt~\cite{hamenstaedt} for classes 
of groups that were later identified as acylindrically hyperbolic groups by Osin~\cite{osin}.

\begin{theorem}[\cite{osin}*{Theorem~8.3}]\label{thm: acylindrically hyperbolic contained in Dreg} The class $\Dreg$ strictly contains the class of acylindrically hyperbolic groups. 	
\end{theorem}

We also mention the following result of Sun~\cite{sun}.  

\begin{theorem}\label{thm: convergence groups are acylindrically hyperbolic} Non-elementary convergence groups are acylindrically hyperbolic. 
\end{theorem}

\begin{remark}
	In the sequel, we show properties $\NBC$ and $\CAF$ for groups in $\Dreg$. Using the difficult Theorem~\ref{thm: acylindrically hyperbolic contained in Dreg} 
	this implies $\NBC$ and $\CAF$ for acylindrically hyperbolic groups. We emphasize, however, that one can deduce $\CAF$ and $\NBC$ directly and 
	quite easily for acylindrically hyperbolic groups. For example, $\CAF$ follows directly from~\cite{osin}*{Corollary~1.5} 
	and the non-amenability of acylindrically hyperbolic groups.  
\end{remark}

We show that groups in $\Dreg$ satisfy the $\wNBC$ property for the trivial reason that the situation in which to apply $\wNBC$ cannot happen. That means, we show the following stronger property for groups in $\Dreg$. Since $\Dreg$ is closed under passing to quotients by finite normal subgroups this will imply that groups in $\Dreg$ satisfy \NBC. 

\begin{defn}
A group $\Gamma$ has property \NBCone~if it satisfies the following: 
If $N\normal\Gamma$ is a normal subgroup and $M<\Gamma$ is a 
commensurated subgroup and $N\cap M=\{1\}$, then 
$N$ is finite or $M$ is finite. 
\end{defn}

It is obvious that \NBCone~implies \wNBC.

\begin{lemma}\label{lem:finite index product}
Let $\Lambda$ be a group with a normal subgroup $N\normal\Lambda$ 
and a commensurated subgroup $M<\Lambda$ such that 
$\Lambda=N\cdot M$ and $N\cap M=\{1\}$. 

To every finite subset $F\subset \Lambda$ we can assign subgroups $N_F<N$ and $M_F<M$ in such a 
way that the following holds: 
\begin{enumerate}
	\item $N_F$ is normalized by $M$; 
	\item $M_F$ is a normal subgroup of finite index in $M$;
	\item $F\subset \Lambda_F:=N_F M$ and $N_FM_F<\Lambda_F$ has finite index;
	\item $M_F$ commutes with $N_F$, thus $N_FM_F\cong N_F\times M_F$;
	\item $F\subset F'$ implies $N_F\subset N_{F'}$ and $M_F\supset M_{F'}$. 
\end{enumerate}
In particular, if $N$ is finitely generated, then $M$ has a finite index subgroup 
that commutes with $N$. 
\end{lemma}

\begin{proof}
For every $F\subset \Lambda$ let $M_F$ be the normal core within $M$ of the subgroup 
\[
	M_F':=\bigcap_{\lambda\in F\cup\{1\}}\lambda M\lambda^{-1}\subset M. 
\]
Since $M$ is commensurated, $M_F'<M$, hence $M_F<M$, are of finite 
index. In other words, $M_F$ is the 
largest normal subgroup of $M$ with the property $\lambda M_F\lambda^{-1}\subset M$ 
for all $\lambda\in F$. 

Let $q\colon\Lambda\to N$ be the map $q(nm)=n$; it is well defined because 
of $N\cap M=\{1\}$. 
Clearly, $M_F'=M_{q(F)}'$, thus $M_F=M_{q(F)}$. Since $N\normal\Lambda$ is normal, 
the commutator $[n,m]=n(mn^{-1}m^{-1})$ lies in $N$. For $m\in M_F=M_{q(F)}$ and 
$n\in q(F)$ we also have $[n,m]=(nmn^{-1})m^{-1}\in M$, hence 
$[n,m]\in M\cap N=\{1\}$. So $M_F$ commutes with the subset $q(F)$. 
Let 
\[
	S:=\bigcup_{m\in M} mq(F)m^{-1}\subset N.
\]
Because of normality $M_F\normal M$ the group $M_F$ commutes with 
the subset $S$ as well. The subgroups $N_F:=\langle S\rangle<N$ and $M_F$ 
satisfy all the required properties. 
\end{proof}

\begin{theorem}\label{thm: groups in Dreg}
Every group in $\Dreg$ is \IRR, \CAF~and \NBCone. 
\end{theorem}

\begin{proof}
The properties $\IRR$ and $\CAF$ are shown in~\cite{thom}*{Theorem~3.4}. Let $\Gamma\in\Dreg$, and let 
$N\normal \Gamma$ and $M<\Gamma$ normal and commensurated subgroups, 
respectively, with trivial intersection. Assume by contradiction 
that both subgroups are infinite. According to \emph{loc. cit.} $N$ is 
non-amenable. Let $N=\{n_1,n_2,\ldots\}$ be an enumeration of $N$. 
We refer to the notation of Lemma~\ref{lem:finite index product} applied to our situation. 
The group $N$ is the increasing union of subgroups $N_{F_i}<N$ 
for $F_i=\{n_1,\ldots, n_i\}$. Hence there is 
some $k\in\bbN$ such that $N':=N_{F_k}$ is non-amenable, thus infinite. 
 
The group $N'$ commutes with the finite index subgroup $M':=M_{F_k}$ 
of $M$. Since $M<\Gamma$ is commensurated, also $M'<\Gamma$ 
is commensurated. By~\cite{thom}*{Lemma~3.3} the restriction maps 
$QH^1(\Gamma,\ell^2(\Gamma))\to QH^1(M', \ell^2(\Gamma))$ and 
$QH^1(N'M', \ell^2(\Gamma))\to QH^1(M',\ell^2(\Gamma))$ to commensurated subgroups are 
injective. Hence the restriction map 
\[QH^1(\Gamma, \ell^2(\Gamma))\to QH^1(N'M', \ell^2(\Gamma))\] 
is 
injective. We reach a contradiction as the latter module 
vanishes by~\cite{thom}*{Theorem~3.4} and the fact that $N'M'$ is a product of two infinite groups. 
\end{proof}

\subsection{Conclusion of proof of Theorem~\ref{T:good-groups}}
\label{sub: proof of properties of groups}

\begin{proof}[Proof of property \CAF]
We refer to the cases in the statement of Theorem~\ref{T:good-groups}. 
\begin{enumerate}
	\item Since all $\ell^2$-Betti numbers of an infinite amenable group vanish by a result of Cheeger--Gromov~\cite{cheeger+gromov}, every group with some non-vanishing $\ell^2$-Betti number is \CAF~according to~\cite{BFS}*{Corollary~1.4}. 
	\item This is exactly~\cite{breuillardetal}*{Theorem~6.8}. 
	\item See Theorem~\ref{thm: groups in Dreg}. For the classes included in $\Dreg$ see Theorems~\ref{thm: acylindrically hyperbolic contained in Dreg} and~\ref{thm: convergence groups are acylindrically hyperbolic}. \qedhere
\end{enumerate}
\end{proof}

\begin{proof}[Proof of property \NBC]
We refer to the classes of groups in the statement of Theorem~\ref{T:good-groups}. All these classes 
are closed under passing to quotients by finite normal subgroups: 

For (1') the closure property follows from the fact that the vanishing of $\ell^2$-Betti 
	numbers is an invariant of virtual isomorphism which can be easily deduced from their basic properties~\cite{lueck-book}. Or one may cite the much more general result that the vanishing of $\ell^2$-Betti numbers of lcsc groups is 
	a coarse invariant~\cite{sauer+schroedl}. 
	
	For (2') one argues as follows: Let $\Gamma<\GL_n(k)$ be a linear group with finite 
	normal subgroup $F<\Gamma$. A linear representation of the algebraic group $\bar\Gamma/K$ where 
	$\bar \Gamma$ is the Zariski closure of $\GL_n(k)$ yields a linear embedding of $\Gamma/K$.  

For (3') the closure property is elementary and left to the reader. 

For (4') we argue as follows: In view of Proposition~\ref{prop: nonamenable groups in Dreg} let $c\colon\Gamma\to\ell^2(\Gamma)$ be an unbounded quasi-1-cocycle. Let $F<\Gamma$ be a finite normal subgroup. Let $\pi\colon \ell^2(\Gamma)\to \ell^2(\Gamma/F)$ be the map of unitary $\Gamma$-representation that sends the vector $\gamma\in \ell^2(\Gamma)$ to $\gamma F\in \ell^2 (\Gamma/F)$. Let $s\colon \Gamma/F\to \Gamma$ be a set-theoretic section of the projection $\Gamma\to\Gamma/F$. 
    Then $\pi\circ c\circ s\colon \Gamma/F\to \ell^2(\Gamma/F)$ is easily seen to be an unbounded quasi-1-coycle. 
    
Thus it suffices to show property \wNBC~for the classes in Theorem~\ref{T:good-groups} which we 
do next. 
\begin{enumerate}
	\item[(1')] Let $\Gamma$ be a group with positive first $\ell^2$-Betti number. Assume by contradiction that $\Gamma$ does not have \NBCone. 
	
	Let $M,N<\Gamma$ be subgroups with $M\cap N=\{1\}$ such that $N\normal \Gamma$ is normal 
	and $M$ is commensurated by $\Gamma$. 
	Asssume that neither $M$ nor $N$ are finite. 
    We have to show that 
	$\betti_1(\Gamma)=0$. Since $\Lambda:=NM$ is commensurated by $\Gamma$, the vanishing of 
	$\betti_1(\Lambda)$ would imply this according to~\cite{BFS}*{Corollary~1.4}. 
	By Lemma~\ref{lem:finite index product} the group $\Lambda$ is an ascending union of groups 
	$\Lambda_n$ such that $\Lambda_n$ is virtually a product of infinite groups. 
	A product of infinite groups has vanishing first $\ell^2 $-Betti number because of the 
	K\"unneth formula for $\ell^2$-Betti 
	numbers and the fact that the zeroth $\ell^2$-Betti number of 
	an infinite group vanishes. The same is true for a virtual product, so $\betti_1(\Lambda_n)=0$ for large $n$; see the remark about $\ell^2$-Betti numbers of virtually isomorphic groups above.   
	Finally, $\betti_1(\Lambda_n)=0$ for all $n$ implies $\betti_1(\Lambda)=0$ by~\cite{lueck-book}*{Theorem~7.2 (2)}. 
	\item[(2')] Let $\Gamma<\GL_n(k)$ be a linear group, and let 
$N, M<\Gamma$ subgroups with trivial intersection such that $N\normal\Gamma$ is normal 
and $M<\Gamma$ is commensurated by $\Gamma$.

Let $F\mapsto N_F, M_F$ be an assignment of finite subsets to subgroups as 
in Lemma~\ref{lem:finite index product}. 
By the Noetherian property of the algebraic group $\GL_n(k)$ 
there exists some finite set $F_1 \subset N$ such that 
the Zariski closure $\bar{M}_1$ of $M_{F_1}$ is minimal among 
all finite subsets of $N$. 
Thus, for every $n\in N$, $M_{F_0\cup \{n\}}$ is Zariski dense in $\bar{M}_1$.
Since \[[n,M_{F_0\cup \{n\}}]=1\] we get that $[n,\bar{M}_1]=1$. In particular, 
the finite index subgroup $M_{F_1}$ of $M$ commutes with $N$. This implies 
\wNBC~for linear groups. 
	\item[(3')] As mentioned before, this is~\cite{capraceetal-closure}*{Corollary 18}. 
	\item[(4')] See Theorem~\ref{thm: groups in Dreg}. \qedhere 
\end{enumerate}
\end{proof}

\begin{proof}[Proof of property \IRR]
By the K\"unneth formula for $\ell^2$-Betti numbers and the fact from the previous proof that their vanishing is unaffected by passing to virtually isomorphic groups it follows that groups with positive first $\ell^2$-Betti numbers have property $\IRR$. That groups in $\Dreg$ have \IRR~is contained in Theorem~\ref{thm: groups in Dreg}.
\end{proof}

\subsection{Examples of groups without property $\NBC$} 
\label{sub: examples_of_groups_that_are_not_in_bbc_}

Among the conditions required for our main result, the \NBC~property is the most opaque. We present two examples of groups that do not have property \NBC. The first example is a wreath product. The second one comes with a lattice envelope that shows the necessity of the \NBC~condition in Theorem~\ref{T:main}.

\begin{example}
For an infinite residually finite group $M$ and a countable group $H$ 
we consider the product group $H^M=\prod_M H$ endowed with the natural (shift) action of $M$. 
Let $N$ be the subgroup in $H^M$ consisting of \emph{periodic} elements, that is, elements having an $M$-stabilizer of finite index in $M$.
$N$ is clearly a countable subgroup on which $M$ acts.
Let $\Gamma=M \ltimes N$.
Observe that $\Gamma=NM$, $N\lhd \Gamma$, $M$ is commensurated by $\Gamma$ and $N\cap M=\{e\}$.
The subgroup $N$ does not commute with any finite index subgroup of $M$, yet any finitely generated subgroup of $N$ does so (cf.~Lemma~\ref{lem:finite index product}).
\end{example}

\begin{example}\label{exa: exotic lattice}
Let $\Lambda$ be an irreducible lattice in $\SL_n(\bbQ_p)\times \SL_n(\bbR)$. 
Let $X$ be the symmetric space of $\SL_n(\bbR)$, and 
let $B$ be the 1-skeleton of the Bruhat-Tits building associated with $\SL_n(\bbQ_p)$. 
Then the action of $\Lambda$ on $B\times X$ can be extended to an action of 
an extension, denoted by $\Gamma$, of $\Lambda$ by the fundamental group $\pi_1(B)$, 
which is an infinitely generated free group, 
on the universal covering $\tilde B\times X$. Note that $T:=\tilde B$ is a tree. 
Moreover, $\Gamma$ is an irreducible lattice in $\Aut(T)\times \SL_n(\bbR)$. 
This type of example originates from the work of Burger--Mozes~\cite{burger+mozes}*{1.8}. 
More general constructions 
in this direction are discussed in Caprace--Monod~\cite{Caprace+Monod:I}*{6.C}. 

The group $\Gamma$ has the properties \BT, \IRR, and \CAF, but it lacks 
property \NBC. Let us show property \CAF~first. 
Let $A<\Gamma$ be a commensurated subgroup. 
The image $p(A)$ under the projection $p:\Gamma\to\Lambda$ is 
an amenable commensurated subgroup of $\Gamma$. 
It follows from Theorem~\ref{T:good-groups} 
that $\Lambda$ has \CAF. Thus $p(A)$ is finite. Since $\pi_1(B)$ satisfies \CAF~by \emph{loc.~cit.}, the intersection 
$A\cap \pi_1(B)$ is finite as well, which implies that $A$ is finite. 

Let $U\subset \Aut(T)$ be the stabilizer of a vertex; it is a compact open subgroup. 
Then it is easy to verify that the normal subgroup $N=\pi_1(B)\normal\Gamma$ and 
the subgroup $M=\Gamma\cap (U\times \SL_n(\bbR))$, which is 
commensurated by Lemma~\ref{L:a-normal}, violate 
property \NBC: No finite index subgroup of $M$ commutes with $N$. 
\end{example}

\section{General facts about locally compact groups and their lattices} 
\label{sec:preliminaries}

In \S\ref{sub: virtual iso} we introduce the analog of weak commensurability or virtual isomorphism of groups for lattice embeddings. In \S\ref{sub:lattices_in_locally_compact_groups} and \S\ref{sub:closed_subgroup_of_finite_covolume} we have a look at broader classes of subgroups in lcsc groups than lattices, discrete subgroups and closed subgroups of finite covolume. In \S\ref{sub:the_amenable_radical} we recall the definition of the amenable radical and study its behaviour under passage to closed subgroups of finite covolume. Finally, in \S\ref{sub: outer automorphism groups} and \S\ref{sub: QI of locally compact groups} we study automorphisms of lcsc groups and their lattices in different categories: outer automorphism groups and quasi-isometry groups. 

\subsection{Virtual isomorphism of lattice embeddings}\label{sub: virtual iso}

We introduce a relative notion of virtual isomorphism (see Definition~\ref{def: commensurable groups}). 

\begin{definition}\label{def: virtual isomorphism}
	Two lattice embeddings $\Gamma\hookrightarrow G$ and $\Gamma'\hookrightarrow G'$ are \emph{virtually isomorphic} 
	if there is a virtual isomorphism from $G$ to $G'$ that restricts to a virtual isomorphism from $\Gamma$ to $\Gamma'$. 
	
	   A lattice embedding is \emph{trivial} if it is virtually isomorphic to the identity homomorphism of a countable discrete group. 
\end{definition}

\begin{remark}
The above definition generalizes Definition~\ref{def: commensurable groups} in the sense that two countable discrete groups $\Gamma$ and $\Gamma'$ are virtually isomorphic if and only if the lattice embeddings $\id\colon\Gamma\to \Gamma$ and $\id\colon\Gamma'\to \Gamma'$ are virtually isomorphic. 
\end{remark}

Spelling out the definition above, if $\Gamma\hookrightarrow G$ and $\Gamma'\hookrightarrow G'$ are virtually isomorphic then 
there are open finite index subgroups $H<G$ and $H'<G'$ and compact normal subgroups $K\normal H$ and 
	$K'\normal H'$ and a commutative square 
	\[\begin{tikzcd}
          	(\Gamma\cap H)/(\Gamma\cap K)\ar[r, hook]\ar[d, "\cong"]& H/K\ar[d, "\cong"]\\
          	(\Gamma'\cap H')/(\Gamma'\cap K')\ar[r,hook] & H'/K'
      \end{tikzcd}\]
      with vertical isomorphisms. 
The horizontal maps in the above diagram are lattice embeddings. This follows from Lemma~\ref{L:open} and Lemma~\ref{L:reducing}. 

\begin{proposition}\label{prop: virtual iso is equivalence relation}
	Virtual isomorphism of lattice embeddings is an equivalence relation. 
\end{proposition}

\begin{proof} The only non-obvious claim is transitivity. Its proof is similar to the proof that virtual isomorphism of groups is transitive. 

Let $j_l\colon \Gamma_l\hookrightarrow G_l$ for $l\in\{1,2,3\}$ be lattice embeddings such 
that $j_1, j_2$ and $j_2, j_3$ are virtually isomorphic. So for $l\in\{1,2,3\}$ there are finite index subgroups 
$H_l<G_l$ and $H_2'<G_2$ and compact normal subgroups $K_l\normal H_l$ and $K_2'\normal H_2'$ and isomorphisms 
\[ H_1/K_1\xrightarrow{\cong} H_2/K_2~\text{ and }~H_2'/K_2'\xrightarrow{\cong} H_3/K_3\]
that restrict to isomorphisms 
\begin{align*}	
(\Gamma_1\cap H_1)/(\Gamma_1\cap K_1) &\cong (\Gamma_2\cap H_2)/(\Gamma_2\cap K_2)~\text{ and } \\
(\Gamma_2\cap H_2')/(\Gamma_2\cap K_2')&\cong (\Gamma_3\cap H_3)/(\Gamma_3\cap K_3),
\end{align*}
respectively. 
The subgroup $H_2'':=H_2\cap H_2'<G_2$ has finite index in $G_2$. Let \[\pr\colon H_2''\to Q_2:=H_2''/(K_2\cap H_2')\] be the projection.  
Let $K_2''$ be the kernel of the composition of quotient maps  
\[ H_2''\to Q_2\to Q_2/\pr(K_2'\cap H_2).\]
Note that $K_2''$ contains $K_2\cap H_2'=K_2\cap H_2''$ and $K_2'\cap H_2=K_2'\cap H_2''$. 
Since $K_2$ and $\pr(K_2'\cap H_2)$ are compact, $K_2''\normal H_2''$ is a compact normal subgroup of $H_2''$. 
Consider the compositions of quotient maps and isomorphisms as above 
\[ f_1\colon H_1\to H_1/K_1\xrightarrow{\cong} H_2/K_2~\text{ and }~g_3\colon H_2'\to H_2'/K_2'\xrightarrow{\cong} H_3/K_3.\]
Similarly, using the inverses of the above isomorphisms we obtain homomorphisms 
\[ g_1\colon H_2\to H_2/K_2\xrightarrow{\cong} H_1/K_1~\text{ and }~f_3\colon H_3\to H_3/K_3\xrightarrow{\cong} H_2'/K_2'.\]
We set $C_i:=g_i(H_2'')$ and $E_i:=g_i(K_2'')\normal C_i$ for $i\in\{1,3\}$; further, let $\tilde C_i<H_i$ and 
$\tilde E_i<H_i$ be the preimages of $C_i$ and $E_i$ under the quotient maps. Then $\tilde C_i$ has finite index 
in $H_i$, thus also in $G_i$, and $\tilde E_i$ is a compact normal subgroup of $\tilde C_i$ for $i\in\{1,3\}$. 
The homomorphism 
\[ \tilde C_i/\tilde E_i\to H_2''/K_2'',~[x]\mapsto [f_i(x)]\]
is an isomorphism for $i\in\{1,3\}$ with inverse map $[y]\mapsto [g_i(y)]$. Since $f_i$ and $g_i$ preserve the lattices, this isomorphism restricts to an isomorphism 
\[ (\Gamma_i\cap \tilde C_i)/(\Gamma_i\cap \tilde E_i)\cong (\Gamma_2\cap H_2'')/(\Gamma_2\cap K_2'')\] 
for $i\in\{1,3\}$. This finishes the proof of transitivity. 
\end{proof}

\subsection{Discrete subgroups of locally compact groups} 
\label{sub:lattices_in_locally_compact_groups}

We collect some important, albeit easy facts used in the proof of the main result. 

\begin{lemma}[cf.~\cite{Caprace+Monod:kacmoody}*{2.C}]\label{L:open}
	Let $\Gamma<G$ be a lattice in an lcsc group. Let $U<G$ be an open subgroup.
	Then $\Gamma\cap U$ is a lattice in $U$.
\end{lemma}

\begin{lemma}[{\cite{Raghunathan:book}*{Theorem~1.13 on p.~23}}]\label{L:reducing}
	Let $G$ be an lcsc group and $N\normal G$ a normal closed subgroup. 
	The projection of a lattice $\Gamma<G$ to $G/N$ is discrete if and only if $\Gamma\cap N$ is a lattice in $N$. 
	If so, then the projection of $\Gamma$ to $G/N$ is also a lattice. 
\end{lemma}

The following statement is our main source of commensurated subgroups. 

\begin{lemma}\label{L:a-normal}
	Let $G$ be an lcsc group and $U<G$ be a compact open subgroup. Then $U$ is commensurated by $G$. 
	For any subgroup $\Lambda<G$ the intersection $\Lambda\cap U$ is commensurated in $\Lambda$.
\end{lemma}
\begin{proof}
	For any $g\in G$, $U\cap gUg^{-1}$ is an open subgroup of the compact group $U$ (resp. in $gUg^{-1}$), 
	hence it has finite index in $U$ (resp. in $gUg^{-1}$).
	The second statement follows by taking $g\in\Lambda$ and taking intersections with $\Lambda$.
\end{proof}

\begin{lemma}[{\cite{Caprace+Monod:II}*{Lemma 2.12}}]\label{L:fg-cg}
	Any lattice envelope of a finitely generated group is compactly generated.
\end{lemma}

\subsection{Closed subgroups of finite covolume} 
\label{sub:closed_subgroup_of_finite_covolume}

The notion of a lattice, that is a discrete subgroup of finite covolume, can be generalized
to closed, but not necessarily discrete, subgroups as follows. 
	A closed subgroup $H<G$ of an lcsc group~$G$ is said to have \emph{finite covolume} in $G$ if
	$G/H$ carries a finite $G$-invariant Borel measure.

If $H$ has finite covolume in $G$, then $G$ is unimodular if and only if 
$H$ is unimodular~\cite{Raghunathan:book}*{Lemma~1.4 on p.~18}.

\begin{lemma}[\cite{Raghunathan:book}*{Lemma~1.6 on p.~20}] \label{L:intermediate}
	Let $G$ be an lcsc group, and $H<L<G$ closed subgroups. 
	Then $H<G$ has finite covolume if and only if both $L<G$ and $H<L$ have finite covolumes.
\end{lemma}

It is well known that a torsion-free lattice in a tdlc 
group is uniform. A mild generalization is given by the following.

\begin{lemma}[\cite{BCGM}*{Corollary~4.11}] \label{L:bounded-clopen}
Let $G$ be a unimodular tdlc group. Let $H<G$ be a closed subgroup of finite covolume. 
Let $m_H$ be the Haar measure of $H$. 
If there is an upper bound on the Haar measures of compact open subgroups in $H$, i.e.
\[
	\sup \bigl\{ m_H(U) \mid U\textrm{\ is\ a\ compact\ open\ subgroup\ in\ } H \bigr\}<\infty,
\]
then $G/H$ is compact.
\end{lemma}

\begin{lemma}\label{lem: ss-no-intermediate}
	Let $G$ be a connected, semi-simple, center-free, real Lie group without compact factors. Every closed subgroup $H<G$ of finite covolume 
	has the form $H=\Delta\times G_2$, 
	where $G_2$ is a direct factor of $G$, $G=G_1\times G_2$, and $\Delta<G_1$ is a lattice. 
\end{lemma}

\begin{proof}
The Lie algebra $\frak{h}$ of the connected component $H^0$ is 
${\rm Ad}(H)$-invariant. 
Since $H$ is Zariski dense in $G$ by Borel's density theorem~\cite{Borel-density}, 
it follows that $\frak{h}$ is an ideal in the semisimple 
Lie algebra $\frak{g}$ of $G$. Let 
\[\frak{g}=\frak{g_1}\oplus \dots\oplus\frak{g_n}\]
be the decomposition into simple Lie algebras. Let 
\[I_1=\{i\in\{1,\dots,n\}\mid \frak{g_i}\cap\frak{h}=\{0\}\},\]
and let $I_2=\{1,\dots, n\}\backslash I_1$. Then $G=G_1\times G_2$ 
with $G_k$ being the Lie subgroup with Lie algebra 
\[\bigoplus_{i\in I_k} \frak{g}_i.\]
Since $\frak{g_i}$ is simple we either have $\frak{g}_i\cap\frak{h}=\{0\}$ 
or $\frak{g}_i\cap\frak{h}=\frak{g}_i$. 
This implies $H=(G_1\times \{1\}\cap H)\times G_2$ where 
$\Delta:=G_1\times \{1\}\cap H$ is discrete. Since $H$ has finite covolume, $\Delta$ is a lattice. 
\end{proof}

\subsection{The amenable radical} 
\label{sub:the_amenable_radical}

\begin{definition}
	The \emph{amenable radical} of an lcsc group~$G$ is the maximal closed normal 
	amenable subgroup. We denote it by $\Radam(G)\normal G$. 
\end{definition}

\begin{lemma}\label{L:Radam-and-fincovol}
	Let $G$ be an lcsc group, and let  
	$H<G$ be a closed subgroup of finite covolume. 
	Then $\Radam(H)=\Radam(G)\cap H$. 
\end{lemma}

This lemma can be deduced from the following result. 

\begin{lemma}[\cite{furstenberg-bourbaki}*{Prop 4.4 and Theorem 4.5}]\label{L:Furst}
	Let $G$ be an lcsc group and $G\acts M$ a continuous action that is minimal and strongly proximal.
	Let $H<G$ be a closed subgroup of finite covolume. 
	Then the restriction of the $G$-action to $H$ is also minimal and strongly proximal.
\end{lemma}
In \cite{furstenberg-bourbaki} Furstenberg was interested in the restriction to lattices, but the proof 
applies to general closed subgroups of finite covolume.

\begin{proof}[Proof of Lemma~\ref{L:Radam-and-fincovol} from Lemma~\ref{L:Furst}]
	The very definition of $\Radam(G)$ yields the inclusion
	\[
		\Radam(G)\cap H<\Radam(H).
	\] 
	For the converse inclusion, we use an equivalent characterization of the amenable radical as the common kernel of all 
	\emph{minimal, strongly proximal} actions on compact metrizable spaces~\cite{Furman-strprox}*{Proposition~7}. 
	Thus for the converse inclusion it suffices to show that given an arbitrary minimal and strongly proximal $G$-action $G\to\Homeo(M)$,
	its restriction to $H$ is also minimal and strongly proximal, because it would show that $\Radam(H)$ acts trivially in every 
	minimal strongly proximal $G$-action, which is the content of Lemma~\ref{L:Furst}.
\end{proof}




\subsection{(Outer) automorphism groups} 
\label{sub: outer automorphism groups}
We describe various results about groups of 
automorphisms and groups of outer automorphisms of certain groups.
We are confident that the main result of this section, Theorem~\ref{thm: outer automorphism groups are finite}, is known to experts, but it does not seem to be in the literature in the appropriate generality. 

\begin{theorem}
\label{thm: outer automorphism groups are finite}
Let $K_1,\dots,K_n$ be number fields, and let $\mathbf{H}_i$ be connected, non-commutative, adjoint, absolutely simple $K_i$-groups. Let 
$S_i\subset \calV_i$ be finite compatible sets of places of $K_i$ as in \S\ref{sub:arithmetic}. Then any group that is (abstractly) commensurable with 
\[ \prod_{i=1}^n \mathbf{H}_i(\calO_i[S_i])\]
has a finite outer automorphism group. 
\end{theorem}

The proof of Theorem~\ref{thm: outer automorphism groups are finite} will be given by the end of this subsection.
This theorem will be used in \S\ref{sec:proof of main result}
together with Lemma~\ref{lem: splitting} to split certain 
extensions of groups.

For a subgroup $B<A$ of a group $A$, let $\Aut(A)$, $\Aut_B(A)$, and $\Out(A)$ denote the automorphism group of $A$, the subgroup of $\Aut(A)$ preserving $B$ set-wise, and the outer automorphism group of $A$. For a $\sigma$-compact lcsc group $G$ we denote the group of continuous automorphisms by $\Auttop(G)$ and the group of  
    continuous outer automorphisms by $\Outtop(G)=\Auttop(G)/\Inn G$. Note that a continuous automorphism of $G$ is a homeomorphism by the open mapping theorem~\cite{bourbaki}*{IX. \S 5.3}. 
     
We begin by stating two auxiliary results which are elementary. We leave the first one to the reader and provide a reference for the second one. 

\begin{lemma}\label{lem: splitting}
Let $B\normal A$ be a normal subgroup in a group $A$. Then the kernel of the natural homomorphism 
$A\to \Out(B)$ is $B\calZ_A(B)$, where $\calZ_A(B)$ is the centralizer of $B$ in $A$. If, in addition, $B$ has trivial center then $B\calZ_A(B)$ is a direct product $B\times \calZ_A(B)$. In particular, the short exact sequence $B\hookrightarrow A\twoheadrightarrow A/B$ splits as a direct product provided $B$ has trivial center and $A\to\Out(B)$ has trivial image. 
\end{lemma}

\begin{lemma}[\cite{wells}]\label{lem: cocycle1}
Let $B\normal A$ be a normal subgroup in a group~$A$. Let $C$ be the center of $B$. Then the 
sequence of groups 
\begin{equation}\label{eq: nonabelian exact}
 1\to Z^1(A/B; C)\xrightarrow{j} \Aut_B(A)\xrightarrow{\pi} \Aut(B)\times \Aut(A/B)
\end{equation}
is exact. Here the left group is the abelian group of $1$-cocycles of $A/B$ in $C$, $j$ maps 
a cocycle $c$ to $a\mapsto c([a])a$, and $\pi(f)=(f\vert_B, [f])$. 
\end{lemma}

%
%
%
%
%
%
%
%
%

The following three results are immediate consequences of 
Lemma~\ref{lem: cocycle1}.

\begin{lemma}\label{lem: automorphisms of products}
Every topological automorphism of a product $G=G^0\times\tdf{G}$ 
of a center-free connected group $G^0$ and a tdlc group $\tdf{G}$ 
is a 
product of a topological automorphism of $G^0$ and a topological automorphism of $\tdf{G}$. 
\end{lemma}

\begin{proof}
We apply Lemma~\ref{lem: cocycle1} to $A=G$ and $B=G^0$. In this situation the map $\pi$ in~\eqref{eq: nonabelian exact} has an obvious right inverse~$s$. Both $\pi$ and $s$ preserve continuity of automorphisms.  
Let $\alpha\colon G\to G$ be a 
continuous automorphism. Then $\alpha(G^0)=G^0$. Since $G^0$ has trivial center, $\pi$ is injective. So $\alpha = s(\pi(\alpha))$. 
\end{proof}


\begin{lemma}\label{lem: restriction finite kernel}
Let $A$ be a group, and let $B\normal A$ be a characteristic subgroup of finite index. If $B$ has a finite center, then the restriction map $\Aut(A) \to \Aut(B)$ has a finite kernel.
\end{lemma}

\begin{proof} \label{lem:ZZ}
Denote by $K$ the kernel of $\Aut(A)=\Aut_B(A) \to \Aut(B)$ and by $C$ the center of $B$. 
Since $\Aut(A/B)$ is finite, it is enough to show that the kernel $K'$ of $K \to \Aut(A/B)$ is finite.
By Lemma~\ref{lem: cocycle1}, $K'$ is isomorphic to $Z^1(A/B;C)$ which is finite since $A/B$ and $C$ are finite. 
\end{proof}

\begin{cor} \label{cor:outfinite}
If a group $A$ has a characteristic subgroup of finite index that has a finite center and 
a finite outer automorphism group, then also $A$ has a
finite outer automorphism group.
\end{cor}

\begin{proof}
Let $B\normal A$ be characteristic with finite center and 
finite outer automorphism group. 
Since the group $\Aut(A)/B$ surjects onto $\Aut(A)/A=\Out(A)$, it is enough to show that
the former is finite.
This follows by the exactness of 
\[ 1\to\ker\bigl(\Aut(A)\to\Aut(B)\bigr) \to \Aut(A)/B \to \Out(B), \]
as $\Out(B)$ is finite by assumption, and the kernel is finite by Lemma~\ref{lem: restriction finite kernel}. 
\end{proof}

For the purposes of this subsection, we introduce the following definition.

\begin{defn}
An infinite group $A$ is said to be {\em strongly irreducible} if
for every homomorphism $\pi:B_1\times B_2\to A$ with finite cokernel, either $B_1<\ker \pi$ or $B_2<\ker \pi$.

A group is said to be {\em strongly outer finite} if the outer automorphism group of any of its finite index subgroups is finite.
\end{defn}

\begin{lemma} \label{lem:strongirr}
A Zariski dense subgroup of the $K$-points of a connected, non-commutative, adjoint, $K$-simple $K$-algebraic group is strongly irreducible. 
\end{lemma}

\begin{proof}
Let $\mathbf{H}$ be a connected, non-commutative, adjoint, $K$-simple $K$-group, and let 
$\Gamma<\mathbf{H}(K)$ be a Zariski dense subgroup. 

Let $\pi\colon B_1\times B_2\to \Gamma$ be a homomorphism with finite cokernel. 
Since $\Gamma<\mathbf{H}(K)$ is Zariski dense, the Zariski closure of $\im(\pi)$ has to be $\mathbf{H}$ as a finite index subgroup of the connected group $\mathbf{H}$. 
The Zariski closure of $\pi(B_i)$ define $K$-subgroups $\mathbf{H}_i$, $i\in\{1,2\}$ 
that commute in $\mathbf{H}$ and together generate $\mathbf{H}$. 
By the simplicity of $\mathbf{H}$ we conclude that for some $i\in\{1,2\}$,
$\mathbf{H}_i$ is trivial. It follows that $B_i<\ker \pi$.
\end{proof}

\begin{theorem} \label{thm:outfinite}
Let $B$ be a group that is (abstractly) commensurable to a product of a finite family 
of finitely generated, strongly irreducible, strongly outer finite groups. Then 
$B$ has a finite outer automorphism group. 
\end{theorem}

The proof of Theorem~\ref{thm:outfinite} will be preceded by some
preparation.
First we make the following ad-hoc definition, which is not 
standard.

\begin{defn}
A \emph{standard subgroup} of a product $\prod_{i\in I} A_i$ of groups 
is a subgroup of the form $\prod_{i\in I} B_i$ with each $B_i<A_i$ being 
a finite index subgroup.
\end{defn}

\begin{lemma} \label{lem:standradization}
Let $(A_i)_{i\in I}$ be a finite family of strongly irreducible groups, and let 
$(B_j)_{j\in J}$ be a finite family of non-trivial groups with $|J|\ge |I|$. If 
\[ \prod_{j\in J}B_j\xrightarrow{\phi} \prod_{i\in I}A_i\] 
is a monomorphism with finite cokernel, then there are a bijection $f\colon J\to I$ and, for each $j\in J$, a monomorphism $\phi_j\colon B_j\to A_{f(j)}$ with finite cokernel such that $\phi$ is the product of the maps $\phi_j$. In particular, the image $\im(\phi)$ is a standard subgroup. 
\end{lemma}

\begin{proof}
Let $A$ and $B$ be the product of $(A_i)_{i\in I}$ and $(B_j)_{j\in J}$, respectively. 
The proof is by induction on $|I|$.
If $|I|=1$, there is nothing to prove.
Let $|I|>1$. Then $|J|>1$, and we can decompose it non-trivially as $J=J_1\sqcup J_2$.
For $k\in\{1,2\}$ we set $B^{(k)}=\prod_{j\in J_k} B_j$.
We let 
\[ I_k=\{i\in I\mid B_{3-k} <\ker \pi_i\} \quad \mbox{and} \quad A^{(k)}=\prod_{i\in I_k} A_i, \]
where $\pi_i\colon A\to A_i$ is the $i$-th projection.
By the strong irreducibility of the $A_i$'s, $I=I_1\sqcup I_2$ is a partition.
Hence the composition of $\phi\colon B\to A$ with the projection
$A \to A^{(k)}$ factors through $B^{(k)}\cong B/B^{(3-k)}$, thus giving 
a homomorphism $\phi_k\colon B^{(k)}\to A^{(k)}$.
We observe that these $\phi_k$ are 
injective homomorphisms with finite cokernels.
Without loss of generality, $|I_1|\leq |J_1|$.
By induction hypothesis, $|I_1|=|J_1|$
and therefore $|I_2|\leq |J_2|$. Applying the induction hypothesis to both $\phi_k$,
the statement follows.
\end{proof}

Applying Lemma~\ref{lem:standradization} to automorphisms $\phi$, 
we obtain the following. 

\begin{cor} \label{cor:autprod}
Let $(C_i)_{i\in I}$ be a finite family of strongly irreducible groups. 
Then the obvious embedding 
\[\prod_{i\in I} \Aut(C_i)\to \Aut\bigl(\prod_{i\in I} C_i\bigr)\]
has a finite cokernel. 
\end{cor}

\begin{lemma} \label{lem:standchar}
Let $(A_i)_{i\in I}$ be a finite family of strongly irreducible groups, and let $A$ be its product. 
Let $B<A$ be a subgroup of finite index. Then $\prod_{i\in I} (B\cap A_i)$ is a characteristic subgroup of $B$ of finite index. 
\end{lemma}

\begin{proof}
Let us call a subgroup of $B$ that is simultaneously a standard subgroup of $A$ a \emph{substandard subgroup}.  
By Lemma~\ref{lem:standradization}, every $\phi\in\Aut(B)$ preserves substandard subgroups. 
Hence $\prod_{i\in I} (B\cap A_i)$ is characteristic in $B$ as it is the unique maximal element in the collection of substandard subgroups. 
\end{proof}

\begin{proof}[Proof of {Theorem~\ref{thm:outfinite}}]
By assumption, $B$ has a finite index subgroup $C'<B$ that is isomorphic to a finite index subgroup of the product $A$ of a finite family $(A_i)_{i\in I}$ of finitely generated, strongly irreducible, strongly outer finite groups. From now on we will identify $C'$ and its various subgroups with
their images in $A$ under this isomorphism. 
As each $A_i$ is finitely generated, $A$, $C'$, and thus $B$ are finitely generated too. Thus $C'$ has a finite index subgroup $C''$ that is characteristic in $B$ (e.g.~the intersection of all subgroups of $B$ with index $[B:C']$).
By Lemma~\ref{lem:standchar} $C''$ has a characteristic subgroup 
$C$ of finite index which is standard in $A$. That is, $C=\prod_{i\in I} C_i$, where 
$C_i<A_i$ are of finite index. Strongly irreducibility passes to finite index subgroups. So each 
$C_i$ is strongly irreducible. Being characteristic is transitive, so $C$ is characteristic in $B$. 

By Corollary~\ref{cor:autprod},
the obvious embedding $\prod_{i\in I} \Aut(C_i)\to \Aut(C)$
has a finite cokernel. 
By the strong outer finite assumption, $\Out(C_i)$ 
is finite for each $i\in I$. Hence $\Out(C)$ is finite.
By the strong irreducibility assumption $C$ has a 
finite center.
Thus, Corollary~\ref{cor:outfinite} implies that $\Out(B)$ is
finite.
\end{proof}

The following proposition is well known. 

\begin{prop} \label{prop: topological automorphisms}
Let $K$ be a number field, and let $\mathbf{H}$ be non-commutative, connected, adjoint, absolutely simple $K$-group. Let 
$S \subset \calV$ be a finite compatible set of places of $K$ as in \S\ref{sub:arithmetic}. 
Then $\prod_{\nu\in S} \mathbf{H}(K_\nu)$ has a finite group of continuous outer
automorphisms.
\end{prop}

\begin{proof}
By Lemma~\ref{lem:strongirr} each group $\mathbf{H}(K_\nu)$
is strongly irreducible. 
By Corollary~\ref{cor:autprod}, the natural embedding  
\[\prod_{\nu\in S} \Aut(\mathbf{H}(K_\nu))\xrightarrow{j} \Aut(\prod_{\nu\in S} \mathbf{H}(K_\nu))\] has a finite cokernel. 
Further, if $f\in\im(j)$ is a continuous automorphism then $f$ is a product of continuous automorphisms. This implies that the restriction of $j$  
\[\prod_{\nu\in S} \Auttop(\mathbf{H}(K_\nu))\to \Auttop(\prod_{\nu\in S} \mathbf{H}(K_\nu))\] 
has a finite cokernel, too. 
It is thus enough to show that $\Outtop(\mathbf{H}(K_\nu))$ is finite for every $\nu\in S$. 
Since every continuous field automorphism of $K_\nu$ is trivial on
the closure of $\mathbb{Q}$ over which $K_\nu$ is finite, the field $K_\nu$ has a finite group of continuous field automorphisms. 
Moreover, the group $\mathbf{H}$
has a finite group of algebraic outer automorphisms, given by 
Dynkin diagram automorphisms. 
Thus, $\Outtop(\mathbf{H}(K_\nu))$ is indeed finite by~\cite[Chapter I,~1.8.2(IV)]{Margulis:book}.
\end{proof}

\begin{prop} \label{prop:outlattices}
Let $K$ be a number field, and let $\mathbf{H}$ be a non-commutative, connected, adjoint, absolutely simple $K$-group. Let 
$S \subset \calV$ be a finite compatible set of places of $K$ as in \S\ref{sub:arithmetic}. 
Then $\mathbf{H}(\calO[S])$ is finitely generated, strongly irreducible and strongly outer finite.
\end{prop}

\begin{proof}
By \cite[Chapter~IX,~Theorem~3.2]{Margulis:book} $\mathbf{H}(\calO[S])$ is finitely generated, and by Lemma~\ref{lem:strongirr} $\mathbf{H}(\calO[S])$ is strongly 
irreducible. We are left to prove that $\mathbf{H}(\calO[S])$ is strongly outer finite. To this end, consider a subgroup $\Gamma<\mathbf{H}(\calO[S])$ of finite index. 

The group $\Gamma$ is an irreducible lattice in
$H=\prod_{\nu\in S} \mathbf{H}(K_\nu)$.
Let $A=\Auttop(H)$.  
The conjugation homomorphism $H\to A$ is injective as $H$ is center-free and has a finite 
cokernel by Proposition~\ref{prop: topological automorphisms}. Then $\Gamma\hookrightarrow A$ embeds as a lattice. We claim: 
\begin{enumerate}
	\item The normalizer $N=\calN_A(\Gamma)$ of $\Gamma$ in $A$ is discrete in $A$. 
	\item The conjugation action $N\to\Aut(\Gamma)$ is an isomorphism.  	
\end{enumerate}

Let us first conclude the finiteness of $\Out(\Gamma)$ from $(1)$ and $(2)$. 
Since $N$ contains the lattice $\Gamma$ and $N$ is discrete in~$A$, $N<A$ is also a lattice. Hence $[N:\Gamma]<\infty$. By $(2)$, $\Out(\Gamma)$ is isomorphic to $N/\Gamma$, hence finite. 

Regarding~$(1)$, discreteness will follow once we show that $N$ is countable because $N<A$ is closed. 
Since $\Gamma$ is finitely generated, $\Aut(\Gamma)$ is countable. 
Hence $N/\calZ_A(\Gamma)$ is countable, where $\calZ_A(\Gamma)$ denotes the centralizer of $\Gamma$ in $A$. 
Each factor of $H$ is center-free, so $\calZ_H(\Gamma)$ is trivial by the density of the projection of $\Gamma$ in each factor.
Since $H<A$ is of finite index, $\calZ_A(\Gamma)$ is finite, and the countability of $N$ follows.

To verify~$(2)$, 
it is enough to show that every automorphism of $\Gamma$ induces a continuous automorphism of $H$. This is a consequence of the Mostow-Prasad strong rigidity theorem for $H$~\cite[Theorem~VII~7.1]{Margulis:book}.
\end{proof}

\begin{proof}[Proof of {Theorem~\ref{thm: outer automorphism groups are finite}}]
By Proposition~\ref{prop:outlattices} the groups $\mathbf{H}_i(\calO_i[S_i])$ are finitely generated, strongly irreducible and strongly outer finite. 
Thus the proof follows at once from Theorem~\ref{thm:outfinite}.
\end{proof}

Finally, we need a classical result from the work of Tits and Borel-Tits. 

\begin{theorem}\label{thm: finite index in automorphisms of building}
Let $K$ be a number field, and let $\nu$ a finite place of $K$. 
Let $\mathbf{H}$ be a non-commutative, connected, adjoint, absolutely simple $K$-group of $K_\nu$-rank~$\ge 2$. Then the automorphism group of the Bruhat-Tits building associated to $\mathbf{H}$ and $\nu$ contains $\mathbf{H}(K_\nu)^+$ as a subgroup of finite index. 
\end{theorem}

\begin{proof}
In higher rank the automorphism group of the associated Bruhat-Tits building $X$ is isomorphic 
to $\Aut(\mathbf{H}(K_\nu))$ which is isomorphic to an extension of the group 
$\Aut(\mathbf{H})(K_\nu)$ of algebraic automorphisms over $K_\nu$ by a subgroup of field automorphisms of $K_\nu$. 
This follows for example from~\cite{bader+caprace+lecureux}*{Proposition~C.1} and the fact
that a local field of zero characteristic has finitely many continuous field automorphisms.
\end{proof}

\subsection{Quasi-isometries of locally compact groups}\label{sub: QI of locally compact groups}

Recall that a map between metric spaces $f\colon (X,d_X)\to (Y,d_Y)$ is an 
\emph{$(L,C)$-quasi-isometry} if
	\[
		\frac{1}{L}\cdot d_X(x,x')-C< d_Y(f(x),f(x')) < L\cdot d_X(x,x')+C
	\]
	for all $x,x'\in X$ 
	and for every $y\in Y$ there is $x\in X$ so that 
	$d_Y(y,f(x))<C$. 
	Two maps $f,g\colon  (X,d_X)\to (Y,d_Y)$ have \emph{bounded distance} if 
	\[\sup_{x\in X} d_Y(f(x),g(x))<\infty.\] 
	The \emph{quasi-isometry group} $\QI(X,d_X)$ of $(X, d_X)$ is the 
	group (under composition) of equivalence classes of quasi-isometries $(X,d_X)\to (X,d_X)$ under bounded distance. 
		
 	\begin{lemma}\label{lem: QI of lcsc group}
		Let $G$ be an lcsc group containing a closed subgroup $H<G$ such that $G/H$ is compact, the projection $G\to H\bs G$ has locally a continuous cross-section, and $H$ is compactly generated. Let $H\to \Isom(X,d)$ be a properly discontinuous cocompact action on some proper geodesic metric space $(X,d)$. Then the natural homomorphism 
		\[
			\phi\colon H\to \Isom(X,d)\rightarrow \QI(X,d) 
		\]
        extends to a homomorphism $\bar\phi\colon G\to \QI(X,d)$ with the following property.  
		There exist constants $L,C>0$ so that each $\bar\phi(g)$, $g\in G$, is 
		represented by a $(L,C)$-quasi-isometry $f_g\colon X\to X$ such that 
					for every bounded set $B\subset X$ there is a neighbourhood 
			of the identity $V\subset G$ with 
			\begin{equation}\label{eq: quantitative statement in QI lemma}
							\forall_{g\in V}\forall_{x\in B}~~d(f_g(x),x)\le C.
						\end{equation}
						In particular, the isometric action of $H$ on $X$ extends to a quasi-action of $G$ on $X$. 
	\end{lemma}
	
	\begin{remark}\label{rem: continuous local cross section}
    The projection $G\to H\bs G$ of an lcsc group $G$ to the coset space of a closed subgroup $H<G$ 
    has locally a continuous cross-section if $G$ is finite-dimensional with respect to covering dimension~\cite{karube}*{Theorem~2}. 
   	\end{remark}

	\begin{proof}[Proof of Lemma~\ref{lem: QI of lcsc group}]
		Since $H$ is compactly generated, and $H\bs G$ is compact, $G$ is also compactly generated.
		Thus $G$ and $H$ possess left-invariant word metrics that are unique up to quasi-isometry. The inclusion $j\colon H\to G$ is a quasi-isometry. We construct a quasi-inverse $j'\colon G\to H$ of $j$, which is unique up to bounded distance, in the 
		following way. 
		
		Let $\pr\colon G\to H\bs G$ be the projection. Let $s\colon U_0\to G$ be a local continuous  cross-section of $\pr$ defined on some compact neighbourhood $U_0\subset H\bs G$ of $H1$ such that $s(H1)=1$. 
		We can extend $s$ to a global cross-section $s\colon H\bs G\to G$ with relatively compact image $s(H\bs G)$, which is not necessarily continuous. Then 
		\[ j'\colon G\to H,~g\mapsto g\cdot s(\pr(g))^{-1}\]
		is a quasi-inverse of $j$.

		The assignment $f\mapsto j\circ f\circ j'$ yields an isomorphism $j_\ast\colon \QI(H)\to \QI(G)$. Let $x_0\in X$. By the \v Svarc-Milnor theorem the map $h\to hx_0$ is a quasi-isometry $f\colon H\to X$. We pick a quasi-inverse $f'\colon X\to H$ of $f$. Similarly, we obtain an isomorphism $f_\ast\colon \QI(H)\to \QI(X)$. Let $H\to \QI(H)$ be the homomorphism that sends a group element to its left translation. Similarly, $G\to \QI(G)$. 
		
		It is easy to see that the map   	
		$\phi$ coincides with the composition 
		\[ H\to \QI(H)\xrightarrow{f_\ast} \QI(X).\]
		We define $\bar\phi$ to be the composition 
		\[ G\to \QI(G)\xrightarrow{j_\ast^{-1}} \QI(H)\xrightarrow{f_\ast} \QI(X).\]
		Clearly, $\bar\phi$ extends $\phi$. Let $l_g^G\colon G\to G$ denote left translation by $g\in G$ in $G$. 
		Similarly, let $l_h^H\colon H\to H$ be left translation by $h\in H$. 
		Then  
		$\bar\phi(g)$ is represented by the quasi-isometry 
		\[ f_g:= f\circ j'\circ l_g^G\circ j\circ f'.\]  
		Clearly, there are constants $C,L>0$ only depending on $j, j', f$, and $f'$ such that $f_g$ is a $(L,C)$-quasi-isometry. 

        Let $B\subset X$ be a bounded subset. Let 
        \[ D:= \sup_{x\in X}d\bigl(x, f(f'(x))\bigr)=\sup_{x\in X}d\bigl(x, f'(x)x_0\bigr)<\infty.\] 
        The set 
        \[ \bigl\{ f'(x)\mid x\in B\bigr\}\subset H \]
        is bounded, thus relatively compact. Let $K\subset H$ be its closure. The continuous 
        map 
        \[ H\times G\to H\bs G,~(h,g)\mapsto \pr(gh)\]
        sends $H\times\{1\}$ to $H1$. By continuity and compactness of $K$ there is an 
        identity neighbourhood $V_0\subset G$ such that $K\times V_0$ is mapped to $U_0\subset H\bs G$. Then the function 
        \[ \psi\colon K\times V_0\to \bbR,~(h,g)\mapsto d\bigl( ghs(\pr(gh))^{-1}x_0, hx_0\bigr)\]
        is continuous and maps $K\times \{1\}$ to $0$. 
        By continuity and compactness of $K$ there is an identity neighborhood $V\subset V_0$ 
        such that $\psi(K\times V)\subset [0,1]$. 
        
        Now let $x\in B$ and $g\in V$. Let $h=f'(x)\in K$. 
        Then 
        \begin{align*}
        	d\bigl( f_g(x), x\bigr) =  d\bigl( f(j'(gj(f'(x)))), x) 
        							&=  d\bigl( j'(gf'(x))x_0, x)\\
        							&\le   d\bigl( j'(gh)x_0, hx_0)+D\\
        							&=  d\bigl( ghs(\pr(gh))^{-1} x_0, hx_0)+D\\
        							&\le D+1.
        	    	\end{align*}

        Upon 
		replacing the constant $C$ by $\max\{C,D+1\}$, the statement follows. 	
\end{proof}

\section{Arithmetic lattices, tree extensions, and the arithmetic core} 
\label{sec:arithmetic_case}

In \S\ref{sub:arithmetic} we introduce 
the notion of tree extensions of $S$-arithmetic lattices, which appears in our main result, Theorem~\ref{T:main}. In \S\ref{sub:reduction_a_product} we discuss an arithmeticity result for lattices in a product of a semisimple Lie group and a tdlc group, which identifies the second case in Theorem~\ref{T:main}. 

\subsection{Arithmetic lattices and tree extensions} 
\label{sub:arithmetic}
 \begin{setup}\label{setup:arith}
	Let $K$ be a number field, $\calO$ its ring of integers, and let 
	$\mathbf{H}$ be a connected, noncommutative, absolutely simple adjoint $K$-group. 
	Let $\calV$ be the set of inequivalent valuations (places) of $K$, 
	let $\calV^\infty$ denote the archimedean ones, and $\calV^\fin=\calV-\calV^\infty$ 
	the non-archimedean ones (finite places).  
	For $\nu\in\calV$ we denote by $K_\nu$ the completion of $K$ with respect to $\nu$; it is a local field.
	Let $S\subset \calV$ be a finite subset of places that is \emph{compatible with $\mathbf{H}$} in the sense 
	of~\cite{BFS-adelic}. Explicitly, this means the following: 
	\begin{enumerate}
\item For every $\nu\in S$, $\mathbf{H}$ is $K_\nu$-isotropic.
\item $S$ contains all $\nu\in\calV^\infty$ for which $\mathbf{H}$ is $K_\nu$-isotropic. 
\item $S$ contains at least one finite and one infinite place. 
\end{enumerate}
	Let $S^\infty=S\cap \calV^\infty$ and $S^\fin=S\cap \calV^\fin$. 
	Let $\calO[S]\subset K$ be the ring of $S$-integers.
\end{setup}

For $\nu\in\calV$ let $\mathbf{H}(K_\nu)^+<\mathbf{H}(K_\nu)$ be the normal 
subgroup as defined in~\cite{borel+tits-abstract}*{Section~6}. 
Since $K_\nu$ is perfect, the group $\mathbf{H}(K_\nu)^+$ is the subgroup generated by all 
unipotent elements and it has finite index in $\mathbf{H}(K_\nu)$. 
If $\nu\in \calV^\infty$, then 
$\mathbf{H}(K_\nu)^+$ is just the connected component of the identity in the real Lie group $\mathbf{H}(K_\nu)$. 
Define  
\begin{equation*}
	\mathbf{H}_{K,S}:=\prod_{\nu\in S}\mathbf{H}(K_\nu)\qquad 
	\mathbf{H}_{K,S}^+:=\prod_{\nu\in S}\mathbf{H}(K_\nu)^+.
\end{equation*}
The quotient $\mathbf{H}_{K, S}/\mathbf{H}_{K,S}^+$ is finite.

By reduction theory after Borel and Harish-Chandra the diagonal 
embedding realizes $\mathbf{H}(\calO[S])$ is a lattice in $\mathbf{H}_{K,S}$. 
Note that 
\[
	\mathbf{H}_{K,S}=\mathbf{H}_{K,S^\infty}\times \mathbf{H}_{K,S^\fin}
\]
is the splitting into a semi-simple real Lie group and a totally 
disconnected locally compact group. 

\begin{remark}\label{rem: meaning of S-arithmetic}
If $S$ is finite set of places compatible with $\mathbf{H}$, then $\mathbf{H}(\calO[S])$ is a 
lattice in a product of at least two lcsc groups, one of them being a Lie group and one of them being a tdlc group. 
\end{remark}

We now exhibit a generalization of the previous type of lattices.

\begin{defn}\label{defn: generalized S-arithmetic lattice embedding}
Let $S$ be finite. Let $S^\fin_1$ denote the finite places $\nu\in S^\fin$
such that the $K_\nu$-rank of $\mathbf{H}$ is $1$; denote by $T_\nu$ the associated
Bruhat--Tits tree and by $\Aut(T_\nu)$ the tdlc group of automorphisms of this simplicial tree. 
Any closed intermediate group  
\[
	\mathbf{H}_{K,S}^+=\prod_{\nu\in S} \mathbf{H}(K_\nu)^+ < 
	H^* < \prod_{\mu\in S-S^\fin_1}\mathbf{H}(K_\mu)\times \prod_{\nu\in S^\fin_1}\Aut(T_\nu)
\]
is called a \emph{tree extension} of $\mathbf{H}_{K,S}$. 
\end{defn}

Since $\mathbf{H}_{K,S}^+=\prod_{\nu\in S} \mathbf{H}(K_\nu)^+ < 
	H^*$ is cocompact, the 
	inclusion of a subgroup commensurable to $\mathbf{H}(\calO[S])$ into a tree extension 
of $\mathbf{H}_{K,S}$ is a lattice embedding by Lemma~\ref{L:intermediate}. 

\begin{definition}\label{def: definition of S-arithmetic lattice up to tree extension}
Retaining Setup~\ref{setup:arith}, we call the inclusion of a subgroup commensurable to $\mathbf{H}(\calO[S])$ into $\mathbf{H}_{K,S}$ or into a tree extension 
of $\mathbf{H}_{K,S}$ an \emph{$S$-arithmetic lattice (embedding)} or an \emph{$S$-arithmetic lattice (embedding) up to tree extension}, respectively. 
\end{definition}

\begin{remark}\label{rem: tree extension up to finite index}
Denote by $X_\nu$ the Bruhat--Tits buildings or the symmetric space of $\mathbf{H}(K_\nu)$.
For $\nu\in S-S^\fin_1$ the group $\Isom(X_\nu)$ contains $\mathbf{H}(K_\nu)^+$ as a subgroup
of finite index by Theorem~\ref{thm: finite index in automorphisms of building}. 
Further, $\prod_{\nu\in S}\Isom(X_\nu)$ is a finite index subgroup of 
$\Isom(\prod_{\nu\in S} X_\nu)$. This follows from a generalization of the de Rham decomposition~\cite{Caprace+Monod:I}*{Theorem~1.9}. Thus any closed intermediate subgroup
\[
	\mathbf{H}_{K,S}^+<H^* < \Isom(\prod_{\nu\in S} X_\nu)
\]
is, up to passing to a finite index subgroup, a tree extension of $\mathbf{H}_{K,S}$.  
\end{remark}

\subsection{Arithmetic core theorem}
Next we state an arithmeticity result for lattices in products that provides a key step in the proof of Theorem~\ref{T:main}.

\begin{theorem}[Arithmetic Core Theorem]\label{T:arithmetic-core}\hfill{}\\
		Let $H$ be a connected, center-free, semi-simple, real Lie group 
		without compact factors,
		$D$ be a tdlc group and $\Gamma<H\times D$ be a lattice. Assume that 
		\begin{enumerate}
			\item the projection $\Gamma\to D$ has a dense image,
			\item the projection $\Gamma\to H$ has a dense image,
			\item the projection $\Gamma\to H$ is injective,
			\item $D$ is compactly generated.
		\end{enumerate}
		Then there exist number fields $K_1,\dots,K_n$, connected, non-commutative, adjoint, 
		absolutely simple $K_i$-groups $\mathbf{H}_i$, 
		finite sets $S_i\subset \calV_i$ of places of $K_i$ compatible with $\mathbf{H}_i$ as in 
		\S\ref{sub:arithmetic} with the following properties: 
		There is a topological isomorphism
		\begin{equation*}
			H \cong \prod_{i=1}^n (\mathbf{H}_i)_{K_i,S_i^\infty}^0
		\end{equation*}
		and a continuous epimorphism $D\twoheadrightarrow Q$ with compact kernel such that $Q$ is a closed 
		intermediate subgroup  
		\[
			\prod_{i=1}^n (\mathbf{H}_i)_{K_i,S_i^\fin}^+<Q<\prod_{i=1}^n (\mathbf{H}_i)_{K_i,S_i^\fin}
		\]
		and the image of $\Gamma$ in $H\times Q$ 
		is commensurable with the image of 
		\[
			\prod_{i=1}^n \mathbf{H}_i(\calO_i[S_i])\overto{} 
			\prod_{i=1}^n (\mathbf{H}_i)_{K_i,S_i^\infty}\times\prod_{i=1}^n (\mathbf{H}_i)_{K_i,S_i^\fin}.
		\]
\end{theorem}

This result is quite close to Theorem 5.20 in the paper \cite{Caprace+Monod:II} of Caprace and Monod. In fact, if $H$ is assumed to be simple instead of semi-simple, Theorem~\ref{T:arithmetic-core} is essentially 
their theorem. 
In our companion paper~\cite{BFS-adelic} we prove Theorem~\ref{T:arithmetic-core} deducing it from a more general statement, 
in which $D$ is not assumed to be compactly generated.
In this more general case the sets $S_i$ of places of $K_i$ might be infinite, 
and the gap between $\mathbf{H}^+_{K,S}$ and $\mathbf{H}_{K,S}$ becomes large.

In the current paper we use the above arithmetic core theorem~\ref{T:arithmetic-core} in  
\S\ref{sec:proof of main result}, Step 3, case (iii).
At this point in the proof we do not know yet that $\Gamma$ (called $\Gamma_3$ there) 
is finitely generated, but do know compact generation of $D$ (called $\tdf{G}_3$ there)
inherited from the compact generation assumption on $G$.

\medskip

For reader's convenience we sketch the idea of the proof of Theorem~\ref{T:arithmetic-core}
as it appears in~\cite{BFS-adelic}. Apart from step (1) below 
our approach differs from that taken by Caprace and Monod in \cite{Caprace+Monod:II}.
In the following sketch we ignore some important details for the sake of transparency.

\begin{enumerate}
	\item 
	Choose a compact open subgroup $U<D$, and observe that the projection $\Delta$ of 
	$\Gamma_U:=\Gamma\cap (H\times U)$ 
	to $H$ is a lattice in $H$, commensurated by the projection of all of $\Gamma$ to $H$.
	
	If $\Delta$ is irreducible, Margulis' commensurator super-rigidity and arithmeticity theorems 
	provide a number field $K$ and a simple $K$-algebraic group $\mathbf{H}$, so that 
	the semi-simple real Lie group $H$ is locally isomorphic to $\mathbf{H}_{K,S^\infty}$ and $\Delta$ is commensurable
	to the arithmetic lattice $\mathbf{H}(\calO)$; here $S^\infty$ is the set of infinite places $\nu\in\calV^\infty$
	for which $\mathbf{H}(K_\nu)$ is non-compact. 
	If $\Delta$ is reducible, it is commensurable to a product \[\Delta_1\times\cdots \times\Delta_n<H_1\times\cdots\times H_n\]
	where $\Delta_i<H_i$ are irreducible lattice, leading to fields $K_i$ and $K_i$-simple groups $\mathbf{H}_i$, $i=1,\dots,n$.
	For clarity we continue with $n=1$ case.
	
	\item
	We now view $\Gamma$ as a subgroup of the commensurator ${\rm Commen}_H(\mathbf{H}(\calO))$ of the arithmetic 
	lattice $\mathbf{H}(\calO)\simeq\Delta$. 
	If one glosses over the difference between the simply-connected and adjoint forms of $\mathbf{H}$
	(such as $\SL_d$ and $\PGL_d$) then the above commensurator is the subgroup of rational points
	$\mathbf{H}(K)$, and therefore $\Gamma<\mathbf{H}(K)$.
	Define $S^\fin$ to be the set of those non-archimedean places $\nu\in \calV^\fin$ for which
	the image of $\Gamma$ in $\mathbf{H}(K_\nu)$ is unbounded.
	(For example, for $\SL_d(\bbZ)<\Gamma<\PSL(\bbQ)$ the set $S^\fin$ would consist of primes $p$ that appear 
	with arbitrarily high powers in denominators of entries of $\gamma\in\Gamma$).
	Then $\Gamma$ is commensurable to a subgroup of $\mathbf{H}(\calO[S])$.
	
	\item
	For $\nu\in S^\fin$ the image of $\Gamma$ in the tdlc group $\mathbf{H}(K_\nu)$ is not precompact.
	It can be showed using Howe-Moore's theorem (alternatively, use~\cite{Prasad}) that this implies 
	that the image of $\Gamma$ is dense in (an open subgroup of finite index in) $\mathbf{H}(K_\nu)$.
	One can even show that the image of $\Gamma$ is dense in (an open subgroup of finite index in) $\mathbf{H}_{K,S^\fin}$.
	
	\item
	One now considers the closure $L$ of the diagonal imbedding of $\Gamma$ in $D\times \mathbf{H}_{K,S^\fin}$.
	Using the fact that the projections to both factors are dense and that the closure of $\Delta\simeq \mathbf{H}(\calO)$
	in both projections is open compact, one shows that $L$ is a graph of a continuous epimorphism $D\to \mathbf{H}_{K,S^\fin}$
	with compact kernel.
	
	\item
	Finally, the fact that $\Gamma<H\times D$ is a lattice, implies that its image, contained in the $S$-arithmetic lattice
	$\mathbf{H}(\calO[S])$, is a lattice in $H\times \mathbf{H}_{K,S^\fin}\simeq \mathbf{H}_{K,S}$.
	Thus $\Gamma\simeq \mathbf{H}(\calO[S])$.
	
	If $D$ is compactly generated, then so is $\mathbf{H}_{K,S^\fin}$, which implies that $S^\fin$ (hence also all of $S$) is finite.
	It also follows that $\Gamma$ is finitely generated.
\end{enumerate}

\section{Proof of Theorem~\ref{T:main}} 
\label{sec:proof of main result}

The starting point of the proof of Theorem~\ref{T:main} is a consequence of Hilbert's 5th problem 
which was observed by Burger--Monod~\cite{burger+monod}*{Theorem~3.3.3}.

\begin{theorem}\label{T:lcsc-mod-ramen}
	Every locally compact group $H$ contains an open, normal,  finite index subgroup $H'$ containing $\Radam(H)$ such that 
	the quotient $H'/\Radam(H)$ is topologically isomorphic to a direct product 
	of a connected, center-free, semi-simple, real Lie group without compact factors 
	and a tdlc group with trivial amenable radical. 
\end{theorem}

In fact, $H'=\ker\left[H\to \Out\left((H/\Radam(H))^0\right)\right]$ is the kernel of the homomorphism to the (finite) outer automorphism
group of the semi-simple, center-free, real Lie group $(H/\Radam(H))^0$ --  the connected component of the identity of
the lcsc group $H/\Radam(H)$.

\subsection*{Step 1. Reduction to a lattice in a product} 
\label{sub:reduction_a_product}

It is only the very first step of the proof where we shall take advantage of property~\CAF. 
\begin{theorem}\label{T:R-compact} 
   The amenable radical of a lattice envelope of a group with property~\CAF\ 
   is compact. 
\end{theorem}

\begin{proof}
Let $G'<G$ be an open, normal, finite index subgroup so that 
\[
	G_1:=G'/\Radam(G)\cong \ssf{G}_1\times\tdf{G}_1
\]
as in Theorem~\ref{T:lcsc-mod-ramen}, namely $\ssf{G}_1$ is a connected, center-free, semi-simple, real Lie group without compact factors,
and $\tdf{G}_1$ is a tdlc group with $\Radam(\tdf{G}_1)=\{1\}$. 
Pick a compact open subgroup $U<\tdf{G}_1$. We consider the following commutative diagram: 

\[\begin{tikzcd}
   \overbrace{\Gamma\cap p_G^{-1}(j(\{1\}\times U))}^{=:A}\arrow[r,hook]\arrow[d, hook'] & p_G^{-1}(j(\{1\}\times U))\arrow[r]\arrow[d,hook'] & j(\{1\}\times U)\arrow[d,hook']\\
   \Gamma\arrow[r, hook]  &G\arrow[r, "p_G"] \arrow[r]& G/\Radam(G)\\
   \Gamma\cap G'\arrow[r, hook]\arrow[u,hook] & G'\arrow[r, "p_{G'}"]\arrow[u, hook] & G'/\Radam(G)\arrow[r,"\cong"]\arrow[u, hook] &\ssf{G}_1\times \tdf{G}_1\arrow[lu, "j"', bend right]
\end{tikzcd}\]

The arrows are the obvious inclusions and projections. Moreover, 
$j$ is defined by requiring commutativity. 

First we show that $A$ is commensurated by $\Gamma$ which is 
equivalent to $j(\{1\}\times U)$ being commensurated by $G/\Radam(G)$ by 
Lemma~\ref{lem: permanence commensurated}. Let $g\in G/\Radam(G)$. Since 
$G'$ is normal in $G$ and thus $G'/\Radam(G)$ is normal $G/\Radam(G)$, conjugation 
by $g$ is a continuous automorphism of $j(\ssf{G}_1\times \tdf{G}_1)\cong \ssf{G}_1\times \tdf{G}_1$. By Lemma~\ref{lem: automorphisms of products} this automorphism is the product of continuous 
automorphisms $\ssf{c}_g\colon \ssf{G}_1\to\ssf{G}_1$ and 
$\tdf{c}_g\colon \tdf{G}_1\to\tdf{G}_1$. Hence conjugation by $g$ maps 
$j(\{1\}\times U)$ to $j(\{1\}\times \tdf{c}_g(U))$. Since the subgroup $\tdf{c}_g(U)<\tdf{G}_1$ is open and compact, the intersection 
$\tdf{c}_g(U)\cap U$ has finite index in~$U$. This implies that 
$j(\{1\}\times U)$ is commensurated by $G/\Radam(G)$ from which we conclude that  
$A$ is commensurated by $\Gamma$. 

Let $G'_U<G'$ be the preimage of $\ssf{G}_1\times U$ under $p_{G'}$, and 
let 
\[ K:=\ker\bigl( G'_U\xrightarrow{p_{G'}} \ssf{G}_1\times U\to \ssf{G}_1\bigr).
\]
The group $K$ is an extension of the amenable group $\Radam(G)$
by the compact group $U$. Hence $K$ is amenable. The group $\Gamma\cap G_U'$ 
is a lattice in $G_U'$ by Lemma~\ref{L:open}.

We now invoke a deep result of Breuillard and Gelander~\cite[\S 9]{BG-Top-Tits}.
It follows from~\cite[Theorem 9.5]{BG-Top-Tits} and $K$ being a closed amenable subgroup that the projection  
of the lattice $\Gamma\cap G_U'$ in $G'_U/K$ is discrete. The result of Breuillard and Gelander is a generalization of Auslander's theorem where $K$ is a solvable Lie group. 

By Lemma~\ref{L:reducing}, the group 
\[A=\Gamma\cap p_{G'}^{-1}(j(\{1\}\times U))=K\cap (\Gamma\cap G_U')\]
is a lattice in $K$, thus is amenable. Note that we used 
$p_{G}^{-1}(j(\{1\}\times U))=p_{G'}^{-1}(j(\{1\}\times U))$ here which follows from the diagram. 
On the other hand $A$ is commensurated by $\Gamma$ and hence finite. This implies that $K$ and its closed subgroup $\Radam(G)$ are compact. 
\end{proof}

Let us summarize the situation after the first step:
\begin{prop}\label{P:step1}
	Let $\Gamma$ have property \CAF,
	and let $\Gamma<G$ be a lattice embedding. Then the amenable 
	radical $K:=\Radam(G)$ of $G$ is compact. 
	By passing
	to an open normal finite index subgroup $G'<G$ containing $K$ and to the finite index subgroup $\Gamma':=\Gamma\cap G'<\Gamma$ and taking 
	quotients by compact or finite normal subgroups
	$G_1=G'/K$ and $\Gamma_1:=\Gamma'/(\Gamma'\cap K)$ one obtains a 
	lattice embedding  
	\[
		\Gamma_1<G_1=\ssf{G}_1\times \tdf{G}_1, 
	\]
	virtually isomorphic to $\Gamma<G$, into a product of a connected, center free, semi-simple real Lie group without compact factors $\ssf{G}_1$
	and a tdlc group with trivial amenable radical $\tdf{G}_1$.
	Furthermore, $G$ is compactly generated if and only if $\tdf{G}_1$ is compactly generated.
\end{prop}


\medskip

\subsection*{Step 2: Projection to the semi-simple factor is irreducible} 
\label{sub:step_2} \hfill

Let $\Gamma_1<G_1=\ssf{G}_1\times \tdf{G}_1$ be a lattice in a product as in Proposition~\ref{P:step1}. Since compact generation of $G$ is assumed in Theorem~\ref{T:main}, the 
group $G_1$ is compactly generated. 
The center-free, semi-simple, real Lie group $\ssf{G}_1$ splits as a direct product of 
the simple factors:
\[
	\ssf{G}_1=S_1\times \cdots \times S_n.
\]
For a subset $J=\{j_1,\dots,j_k\}\subset \{1,\dots,n\}$ we denote by $S_J=S_{j_1}\times\cdots\times S_{j_k}$ 
the sub-product, which can be viewed both as a subgroup and as a factor of the semi-simple group $\ssf{G}_1$.
Set $S_\emptyset=\{1\}$.
Given a subset $J\subset \{1,\dots,n\}$ consider the image of $\Gamma_1$ under the projection
\[
	\pr_J\colon G_1\overto{} \ssf{G}_1\overto{} S_J.
\]
Note that it is possible that $\pr_{\{i\}}(\Gamma_1)$ is dense in $S_i$ for each $i\in \{1,\dots,n\}$,
but $\pr_J(\Gamma_1)$ is discrete for some non-empty proper subset $J\subset \{1,\dots,n\}$.
\begin{lemma}\label{L:max-discrete-proj}
	There is a unique maximal subset $J\subset \{1,\dots,n\}$ for which 
	the projection $\pr_J(\Gamma_1)$ is discrete in $S_J$.
\end{lemma}
\begin{proof}
	It suffices to show that the collection of all subsets $J\subset \{1,\dots,n\}$
	with discrete projection to $S_J$, is closed under union. 
	Let $J, K\subset \{1,\dots,n\}$ be subsets such that
	$\pr_J(\Gamma_1)$ is discrete in $S_J$ and $\pr_K(\Gamma_1)$ is discrete in $S_K$.
	Let $V\subset S_J$, $W\subset S_K$ be open neighborhoods of the identity with
	\[
		\pr_J(\Gamma_1)\cap V=\{1\},\qquad \pr_K(\Gamma_1)\cap W=\{1\}.
	\] 
	View $S_{J\cup K}$ as a subgroup in $S_J\times S_K$.
	Then $U=S_{J\cup K}\cap (V\times W)$ is an open neighborhood of the identity 
	with $\pr_{J\cup K}(\Gamma_1)\cap U=\{1\}$. Therefore $\pr_{J\cup K}(\Gamma_1)$ 
	is discrete in $S_{J\cup K}$. This proves the Lemma.
\end{proof}

\medskip

Let $J\subset \{1,\dots,n\}$ be the maximal subset as in Lemma~\ref{L:max-discrete-proj}. Denote 
\[
	L=S_J=\prod_{j\in J} S_j,\qquad H=S_{J^c}=\prod_{i\notin J} S_i.
\]
We have $\ssf{G}_1=L\times H$, and consider the projection 
\[
	\pr_L\colon G_1=L\times H\times \tdf{G}_1\to L,\qquad \Gamma_0=\pr_L(\Gamma_1). 
\]
Define $\Gamma_2$ to be the kernel of this projection
\[
	\Gamma_2=\Ker(\pr_L:\Gamma_1\overto{}\Gamma_0).
\]
Then $\Gamma_0<L$ and $\Gamma_2<H\times \tdf{G}_1$ are lattices (Lemma~\ref{L:reducing}).
Consider the projections
\[
	p\colon H\times \tdf{G}_1\overto{} H,\qquad q\colon H\times \tdf{G}_1\overto{}\tdf{G}_1.
\]
\begin{lemma}
	The projection $p(\Gamma_2)$ of $\Gamma_2$ to $H$ is dense.
\end{lemma}
\begin{proof}
	Let $H'=\overline{p(\Gamma_2)}$ be the closure of the projection of $\Gamma_2$ to $H$.
	Then $H'\times \tdf{G}_1$ is a closed subgroup of $H\times \tdf{G}_1$ containing 
	the lattice $\Gamma_2$. By Lemma~\ref{L:intermediate}, $H'\times \tdf{G}_1$
	is a closed subgroup of finite covolume in $H\times \tdf{G}_1$.
	It follows that $H'$ has finite covolume in $H$. 
	By Lemma~\ref{lem: ss-no-intermediate} $H'$ has the
	form $H'=\Delta\times H_2$, 
	where $H_2$ is a direct factor of $H$, $H=H_1\times H_2$, and $\Delta<H_1$ is a lattice. 
	In our setting such a splitting has to be trivial: 
	\[
		H'=H_2=H,\qquad H_1=\Delta=\{1\}.
	\]
	Indeed, otherwise the semi-simple group $\ssf{G}_1$
	splits as $\ssf{G}_1=(L\times H_1)\times H_2$,
	where the projection of $\Gamma_1$ to the $L\times H_1$-factor lies in $\Gamma_0\times \Delta$
	which is discrete.
	This contradicts the maximality of $L$ as such a factor.
	This completes the proof of the lemma. 
\end{proof}

We now define an lcsc group $G_2$ to be $G_2=\ssf{G}_2\times \tdf{G}_2$, where 
\[
	\ssf{G}_2=H,\qquad \tdf{G}_2=\overline{q(\Gamma_2)}<\tdf{G}_1.
\]
Being a closed subgroup of a tdlc group $\tdf{G}_1$, the group $\tdf{G}_2$ itself is  tdlc.  The group $\ssf{G}_2=H$ is a connected, center-free, semi-simple, real Lie group.

The group $G_2$ is a closed subgroup of $H\times\tdf{G}_1$ containing a lattice $\Gamma_2<H\times\tdf{G}_1$. 
By Lemma~\ref{L:intermediate}, $\Gamma_2$ forms a lattice in $G_2$, and $G_2=H\times \tdf{G}_2$ is a finite covolume subgroup of $H\times\tdf{G}_1$.
Thus $\tdf{G}_2$ is a subgroup of finite covolume in $\tdf{G}_1$.
Since $\Radam(\tdf{G}_1)=\{1\}$ we deduce that $\Radam(\tdf{G}_2)=\{1\}$ using Lemma~\ref{L:Radam-and-fincovol}. We summarize: 
\begin{prop}\label{P:step2}
Let $\Gamma_1<G_1=\ssf{G}_1\times \tdf{G}_1$ be as in Proposition~\ref{P:step1}.
Then there is a splitting $\ssf{G}_1=L\times H$ and a closed subgroup
$\tdf{G}_2<\tdf{G}_1$ such that by setting $\ssf{G}_2=H$ and $G_2=\ssf{G}_2\times\tdf{G}_2$
we have:
\begin{enumerate}
	\item 
	The projection $\Gamma_0=\pr_L(\Gamma_1)$ of $\Gamma_1$ is a lattice in $L$.
	\item 
	The kernel $\Gamma_2=\Ker(\Gamma_1\overto{\pr_L}\Gamma_0)$ 
	is a lattice in $G_2=\ssf{G}_2\times\tdf{G}_2$.
	\item 
	The projection $\ssf{\pr}_2(\Gamma_2)$ is dense in $\ssf{G}_2=H$.
	\item
	The projection $\tdf{\pr}_2(\Gamma_2)$ is dense in $\tdf{G}_2$.
	\item 
	$\tdf{G}_2<\tdf{G}_1$ is a closed subgroup of finite covolume.
	\item 
	$\tdf{G}_2$ has trivial amenable radical. 
	\item 
	$\tdf{G}_2$ is compactly generated. 
\end{enumerate}
\end{prop}
Most of the conditions needed for an application of Theorem~\ref{T:arithmetic-core} are satisfied
with the exception of injectivity of the projection
\[
	\ssf{\pr}_2: \Gamma_2\overto{}\ssf{G}_2=H.
\]
This is our next topic of concern.

\medskip

\subsection*{Step 3. Identifying the lattice embedding.} 
\label{sub:step_3}\hfill

We still denote the projections of $G_1$ and $G_2$ to their Lie and tdlc factors by 
\[
	\ssf{\pr}_i\colon G_i\to\ssf{G}_i~~\text{ and }~~
	\tdf{\pr}_i\colon G_i\to\tdf{G}_i ~~\text{for $i\in\{1,2\}$.}
\] 
The images of a subgroup in $G_i$ under $\ssf{\pr}_i$ and $\tdf{\pr}_i$ will be indicated by superscripts 
$^{\mathrm{ss}}$ and $^{\mathrm{td}}$, respectively. For a subgroup $H<G$ let 
\[
	\calZ_G(H):=\{ g\in G \mid \forall_{h\in H}~ gh=hg\}
\] 
denote the centralizer of $H$ in $G$. 

Let $U<\tdf{G}_1$ be a compact open subgroup in the tdlc group $\tdf{G}_1$  
such that $\Gamma_1\cap (\{1\}\times U)=\{1\}$. 
We define the following groups: 
\begin{align}\label{eq: definition of groups in step 3}
	N_i &:=\Gamma_i \cap (\{1\}\times \tdf{G}_i)~\text{ for $i\in\{1,2\}$},\notag\\
	M_1 &:= \Gamma_1\cap (\ssf{G}_1\times U),\notag\\
	M_2 &:= \Gamma_2\cap M_1,\\
	\Gamma_3 &:= \Gamma_2/N_2,\notag\\
	M_3 &:= p(M_2)~\text{ where $p\colon \Gamma_2\twoheadrightarrow \Gamma_3$ is the projection},\notag\\
	\ssf{G}_3 &:= \ssf{G}_2,\quad \tdf{G}_3 := \tdf{G}_2/\tdf{N}_2,\quad G_3:=\ssf{G}_3\times\tdf{G}_3.\notag
\end{align} 

\begin{remark}
Since $\tdf{N}_2$ is normalized by the dense image of $\Gamma_2$ in $\tdf{G}_2$ 
and the normalizer of a closed subgroup is closed, the subgroup $\tdf{N}_2$ is normal 
in $\tdf{G}_2$. This justifies the last definition. Moreover, since $\ssf{G}_1$ commutes 
with $N_2$ it follows that $N_2$ is normal in $G_2$. 
Note that since compact generation passes to quotients, $\tdf{G}_3$ is compactly generated. 
\end{remark}

\begin{remark}
At this point, we do not claim triviality of the amenable radical of $\tdf{G}_3$,
even though it will follow from later analysis. 
\end{remark}

\begin{lemma}\label{lem: normalizing}
The group $N_1$ commutes with some finite index subgroup of $M_1$. 
\end{lemma}

\begin{proof}
At this point we want to apply property \NBC~of the original group $\Gamma$. Note that property \NBC~does not pass to finite index subgroups in general, so we have to argue more specifically. We refer to the notation of 
Proposition~\ref{P:step1}. Note that $\Gamma_1=(\Gamma\cap G')/(\Gamma\cap K)$ 
is a subgroup of $\Gamma/(\Gamma\cap K)$. We show that $N_1$ is normal 
in $\Gamma/(\Gamma\cap K)$ and $M_1$ is commensurated by $\Gamma/(\Gamma\cap K)$. 
From this the claim follows since $\Gamma/(\Gamma\cap K)$ has \NBC~by Lemma~\ref{lem: NBC and finite kernel}. 

The group $\{1\}\times \tdf{G}_1$ is a 
topologically characteristic subgroup of 
$G_1=\ssf{G}_1\times \tdf{G}_1=G'/K$ by Lemma~\ref{lem: automorphisms of products}.  
Since $G'/K<G/K$ is normal, $\{1\}\times\tdf{G}_1$ is normal in $G/K$. Since 
$\Gamma_1=(\Gamma\cap G')/(\Gamma\cap K)$ is a normal subgroup of 
$\Gamma/(\Gamma\cap K)$, the group $N_1$ is normal in $\Gamma/(\Gamma\cap K)$. 
The subgroup $\ssf{G}_1\times U<G_1=G'/K$ is commensurated by $G/K$ since 
$G'/K\normal G/K$ is normal and every (topological) automorphism of 
$G_1$ is a product of an automorphism of $\tdf{G}_1$ and one of $\ssf{G}_1$ by Lemma~\ref{lem: automorphisms of products}.  
Since $\Gamma_1\normal \Gamma/(\Gamma\cap K)$ is normal, $M_1$ is commensurated by 
$\Gamma/(\Gamma\cap K)$. 
\end{proof}

By the previous lemma and upon making $U$ smaller we 
may and will assume that $M_1$ itself centralizes $N_1$ and $M_1\cap N_1=\{1\}$.  
It follows that $M_2$ centralizes $N_2$. We record that 
\begin{equation}\label{e:M1Z1}
	M_i\cap N_i=\{1\}\quad\text{and}\quad [M_i,N_i]=\{1\}~~\text{ for $i\in\{1,2\}$.} 
\end{equation}
In particular, the subgroup $M_iN_i<\Gamma_i$ is isomorphic to $M_i\times N_i$. 

\begin{lemma}\label{lem: new lattices}
	The inclusions $\ssf{M}_i\hookrightarrow \ssf{G}_i$ for $i\in\{1,2,3\}$ and $\Gamma_3\hookrightarrow G_3$ are  
	lattice embeddings.   
\end{lemma}

\begin{proof}
By Lemma~\ref{L:open}, $M_1$ is a lattice in $\ssf{G}_1\times U$ since $U<\tdf{G}_1$ is 
an open subgroup. Since $U$ is also compact, the image $\ssf{M}_1$ of $M_1$ in $\ssf{G}_1$ 
is a lattice of $\ssf{G}_1$ as well. 
Similarly, $M_2<\ssf{G}_2\times U$ and $\ssf{M}_2<\ssf{G}_2$ are lattices. Since $\ssf{G}_2=\ssf{G}_3$ by 
definition one easily sees that the images of $M_2$ and $M_3$ in $\ssf{G}_2$ coincide, so $\ssf{M}_3=\ssf{M}_2$. 
In particular, $\ssf{M}_3<\ssf{G}_3$ is a lattice. 

Since $N_2\normal \Gamma_2$ is also normal in $G_2$ (see the remark below~\eqref{eq: definition of groups in step 3}), the quotient $\Gamma_3$ is a lattice in $G_3=\ssf{G}_2\times \tdf{G}_2/\tdf{N}_2\cong G_2/N_2$ by 
Lemma~\ref{L:reducing}. 
\end{proof}

We are finally in a position to identify the lattice embedding $\Gamma_1<G_1$, which is 
virtually isomorphic to the original lattice embedding $\Gamma<G$. 
We distinguish three cases depending on the finiteness of the groups $M_1$ and $M_2$.


\subsubsection*{Case (i): $M_1$ is finite ($\Gamma_1$ is a lattice in a tdlc group)}\hfill{}\\  
In this case, the connected real Lie group $\ssf{G}_1$ without compact factors has a 
finite group $\ssf{M}_1$ as a lattice, thus $\ssf{G}_1$ must be trivial. 
Thus $G_1=\tdf{G}_1$
is a tdlc group with trivial amenable radical that contains 
$\Gamma_1$ as a lattice.
If $\Gamma$ is assumed to have property \BT, the same applies to $\Gamma_1$,
and by Lemma~\ref{L:bounded-clopen}, $\Gamma_1<G_1$ is a uniform lattice. 
Either $G_1$ is discrete which means that the original lattice embedding $\Gamma\hookrightarrow G$ was trivial or we are in case (3) of the main theorem. 

\subsubsection*{Case (ii): $M_1$ is infinite, but $M_2$ 
is finite ($\Gamma_1$ is a lattice in a semi-simple Lie group)}\hfill{}\\  
In this case $\ssf{M}_2$ is finite and a lattice in $\ssf{G}_2$. 
As above we conclude that $\ssf{G}_2=H$ is trivial. 
Therefore $\Gamma_2=N_2$, and 
\[
	\tdf{G}_2=\overline{\tdf{\Gamma}_2}=\overline{\tdf{N}_2}=\tdf{N}_2.
\]
The last equality follows from the discreteness of $\tdf{N}_2$. 

Since $\ssf{M}_1=\ssf{\pr}_1(M_1)$ is a lattice in $\ssf{G}_1=L\times H=L$ and 
a subgroup of the lattice $\Gamma_0$ (see Proposition~\ref{P:step2}), 
$\ssf{M}_1<\Gamma_0$ must have finite index. 
We have a short exact sequence
\[
	1\overto{} \Gamma_2\overto{} \Gamma_1\overto{\pr_L} \Gamma_0\overto{} 1. 
\]
The $\pr_L$-preimage of $\ssf{M}_1$ which equals $\Gamma_2\cdot M_1=N_2M_1\cong N_2\times M_1$ 
is a finite index subgroup of $\Gamma_1$. Since $M_1$ is infinite, the \IRR\ condition forces $\Gamma_2=N_2$ to be finite. So $\tdf{G}_2$ is finite.
As $\tdf{G}_2$ has finite covolume in $\tdf{G}_1$, the latter is compact.
In fact, $\tdf{G}_1$ has to be trivial, because it has a trivial amenable radical.
We conclude that $\Gamma_1$ is a classical lattice in a connected, center-free,
semi-simple, real Lie group $L=\ssf{G}_1=G_1$. This lattice is irreducible
due to assumption \IRR.  
This corresponds to case (1) in the main theorem.

\subsubsection*{Case (iii): $M_2$ is infinite ($\Gamma_1$ is an $S$-arithmetic lattice)}\hfill{}\\  
Recall that the projection $\ssf{M}_2$ of $M_2$ to $\ssf{G}_2=\ssf{G}_3=H$ is a lattice there.
The assumption that $M_2$ is infinite, means that $\ssf{G}_3$ is non-trivial.
The inclusion 
\[
	\Gamma_3<G_3=\ssf{G}_3\times\tdf{G}_3
\]
is a lattice embedding by Lemma~\ref{lem: new lattices}. 
At this point we may and will apply the Arithmetic Core Theorem~\ref{T:arithmetic-core} to 
$\Gamma_3<G_3$ 
to deduce that, up to dividing $\tdf{G}_3$ by a compact normal subgroup $C$, one
has a product of $S$-arithmetic lattices with non-trivial 
semi-simple and totally disconnected factors.
(It will become clear below that $C$ is trivial and that there 
is only one irreducible $S$-arithmetic lattice).
More precisely, there are number fields $K_1,\dots,K_n$, 
absolutely simple $K_i$-algebraic groups $\mathbf{H}_i$, and 
finite sets $S_i\subset \calV_i$ of places of $K_i$ compatible with $\mathbf{H}_i$ 
such that, denoting the connected real Lie groups by $H_i=(\mathbf{H}_i)^+_{K_i,S_i^\infty}=(\mathbf{H}_i)^\circ_{K_i,S_i^\infty}$, we have
\begin{equation*}
	\ssf{G}_3\times \tdf{G}_3/C=\prod_{i=1}^n H_i\times \prod_{i=1}^n Q_i
\end{equation*}
for certain closed intermediate groups
\[
	(\mathbf{H}_i)^+_{K_i, S_i^\fin}<Q_i<(\mathbf{H}_i)_{K_i, S_i^\fin}.
\]
Moreover, $\Gamma_3$ is commensurable to the product of the $S_i$-arithmetic lattices: There are
finite index subgroups $\Gamma_{3,i}<\mathbf{H}_i(\calO_i[S_i])$ such that 
\begin{equation}\label{e:Gamma3}
	\Gamma_{3,1}\times\cdots\times\Gamma_{3,n}< \Gamma_3, ~~[\Gamma_3:\Gamma_{3,1}\times\cdots\times\Gamma_{3,n}]<\infty 
\end{equation}
has finite index. Upon passing to smaller finite index subgroup, we may assume that~\eqref{e:Gamma3} is the inclusion of a normal subgroup. 
By Lemma~\ref{lem: new lattices}, \[\ssf{M}_3<\ssf{G}_3=H_1\times\dots\times H_n\] is a lattice. It contains a product of irreducible lattices $\Delta_i<H_i$ as a finite index subgroup: 
\[
	\Delta_1\times\cdots\times\Delta_n<\ssf{M}_3, ~~[\ssf{M}_3, \Delta_1\times\cdots\times\Delta_n]<\infty. 
\]
Relying on the \NBC\ condition we showed (see (\ref{e:M1Z1})) 
that $M_2<\Gamma_2$ commutes with $N_2$, so
\begin{equation}\label{eq: M2 contained in centralizer}
	M_2<\calZ_{\Gamma_2}(N_2).
\end{equation}
The abelian subgroup $\calZ_{\Gamma_2}(N_2)\cap N_2=\calZ(N_2)$ is,  
being a characteristic subgroup of $\tdf{N}_2\normal \tdf{G}_2$, normal 
in $\tdf{G}_2$. Since $\tdf{G}_2$ has a trivial amenable radical, 
\begin{equation}\label{eq: away from center}
\calZ_{\Gamma_2}(N_2)\cap N_2=\{1\}.
\end{equation}
Hence the restriction $p\vert_{\calZ_{\Gamma_2}(N_2)}$ of 
the quotient map $p\colon \Gamma_2\to \Gamma_3=\Gamma_2/N_2$ is injective. 
The image $p(\calZ_{\Gamma_2}(N_2))$ of the normal subgroup $\calZ_{\Gamma_2}(N_2)$
is normal in $\Gamma_3$ and contains $M_3$. The subgroup 
\[ T_3:=p(\calZ_{\Gamma_2}(N_2))\cap\prod_{i=1}^n\Gamma_{3,i}\]
has finite index in $p(\calZ_{\Gamma_2}(N_2))$ and is normal in $\Gamma_{3,1}\times\dots\times\Gamma_{3,n}$. The commutative diagram below summarizes the various relations between the groups. 

Since 
$\Delta_i$ and $M_3\cap T_3\cap \Gamma_{3,i}$ are commensurable, the normal subgroup $T_3\cap \Gamma_{3,i}\normal \Gamma_{3,i}$ is infinite. Since each $\Gamma_{3,i}$ is an irreducible 
$S_i$-arithmetic lattice, Margulis' normal subgroup theorem implies that 
each $T_3\cap \Gamma_{3,i}<\Gamma_{3,i}$ is of finite index, thus $p(\calZ_{\Gamma_2}(N_2))<\Gamma_3$ is of finite index.  By the 
injectivity of $p\vert_{\calZ_{\Gamma_2}(N_2)}$, the group 
$\calZ_{\Gamma_2}(N_2)$ is thus (abstractly) commensurable with a product of 
$S$-arithmetic lattices. By Theorem~\ref{thm: outer automorphism groups are finite}, $\calZ_{\Gamma_2}(N_2)$ 
has a finite outer automorphism group. Further, the center of $\calZ_{\Gamma_2}(N_2)\normal \Gamma_2$ is trivial since it lies in the amenable radical of $\Gamma_2$ which is trivial 
by Proposition~\ref{P:step2} and Lemma~\ref{L:Radam-and-fincovol}. 

\begin{equation*}
\begin{tikzcd}
M_2\arrow[d, hook]\arrow[rrd, "\cong", shorten >= 0.7em] & & \prod_{i=1}^n\Delta_i\arrow[r, hook]\arrow[d, hook, "f.i."]& \prod_{i=1}^nH_i\arrow[d,equal]\\
\calZ_{\Gamma_2}(N_2)\arrow[r, "\cong"', "p"]\arrow[d, hook, "\normal"] & p(\calZ_{\Gamma_2}(N_2))\arrow[d, hook, "\normal"] & M_3\arrow[l, hook']\arrow[r, hook]\arrow[d, hook] & \ssf{G}_3\times\{1\}\arrow[d,hook]\\
\Gamma_2\arrow[r, two heads, "p"]&\Gamma_2/N_2\arrow[r,equal]& \Gamma_3\arrow[r,hook]& \ssf{G}_3\times \tdf{G}_3\arrow[d, two heads]\\
& T_3\arrow[u,hook']\arrow[r,hook, "\normal"]& \prod_{i=1}^n\Gamma_{3,i}\arrow[u,hook',"f.i."', "\normal"]\arrow[d, "f.i.",hook]\arrow[r,hook]& \prod_{i=1}^n H_i\times Q_i\arrow[d,hook]\\
 & & \prod_{i=1}^n\mathbf{H}_i(\calO_i[S_i])\arrow[r,hook]& \prod_{i=1}^n (\mathbf{H}_i)_{K_i, S_i^\infty}\times (\mathbf{H}_i)_{K_i, S_i^\fin}
\end{tikzcd}
\end{equation*}

Since $N_2$ is the intersection of $\Gamma_2\normal\Gamma_1$ and $N_1\normal\Gamma_1$ the subgroup $N_2$ is normal in $\Gamma_1$ which implies 
that $\calZ_{\Gamma_2}(N_2)$ is normal in $\Gamma_1$. 

By Lemma~\ref{lem: splitting} the finite index subgroup 
$\Gamma_1'=\ker\bigl(\Gamma_1\to \Out(\calZ_{\Gamma_2}(N_2))\bigr)<\Gamma_1$ splits as a direct product: 
\[
	\Gamma'_1\cong \calZ_{\Gamma_2}(N_2) \times (\Gamma'_1/\calZ_{\Gamma_2}(N_2)).
\]
Assumption \IRR\ implies that one of the factors is finite; and by  
$M_2$ being infinite and~\eqref{eq: M2 contained in centralizer}, we obtain that 
$\calZ_{\Gamma_2}(N_2)$ is of finite index in $\Gamma_1$, hence $\Gamma_2<\Gamma_1$ is of finite index. 
So $\Gamma_0$ is finite. Because of $[\Gamma_1:\calZ_{\Gamma_2}(N_2)]<\infty$ and~\eqref{eq: away from center} we obtain that $N_1$ and $N_2$ are finite. 
The group $L$, being a lattice envelope of $\Gamma_0$, is compact, thus $L$ and $\Gamma_0$ are trivial. 

So $\ssf{G}_1=H=\ssf{G}_2=\ssf{G}_3$. 
Using condition \IRR\ we can now also deduce that there was only one irreducible factor, hence $n=1$ in~\eqref{e:Gamma3}.

Furthermore, the group  $N_2\cong \tdf{N}_2$, that is now known to be finite, has to be trivial,
because $\tdf{N}_2$ is normal in $\tdf{G}_2$, while $\Radam(\tdf{G}_2)=\{1\}$.
It follows that $G_3=G_2$ and $\Radam(\tdf{G}_3)=\{1\}$.
In particular, the compact normal subgroup $C\normal \tdf{G}_3$ is actually trivial.

We deduce that there is a number field $K$, 
a connected, non-commutative, adjoint, absolutely simple $K$-groups $\mathbf{H}$, 
and a finite set $S=S^\infty\cup S^\fin$ of places as in \S\ref{sub:arithmetic} so that 
\[
	\Gamma_1=\Gamma_2=\Gamma_3\simeq \mathbf{H}(\calO[S])^+
\] 
is an irreducible $S$-arithmetic lattice in the lcsc group
\[
	G_2=\ssf{G}_2\times\tdf{G}_2=\mathbf{H}^+_{K,S^\infty}\times\mathbf{H}^+_{K,S^\fin}.
\] 
It remains to identify the lcsc group $G_1$.
We have established the triviality of $\Gamma_0$ and of $L$, hence
\[
	\ssf{G}_1=H\times L=H=\ssf{G}_2=\ssf{G}_3=\mathbf{H}^+_{K,S^\infty}.
\]
It remains to identify the totally disconnected component $\tdf{G}_1$ that contains
\[
	\tdf{G}_2=\mathbf{H}^+_{K,S^\fin}=\prod_{\nu\in S^\fin}\mathbf{H}(K_{\nu})^+
\]
as a closed subgroup of finite covolume. 


\medskip

\subsection*{Step 4: Lattice envelope $\tdf{G}_1$ of $\tdf{G}_2=\mathbf{H}^+_{K,S^\fin}$}\hfill{}

Let us enumerate the elements in $S^\fin=\{\nu_1,\dots, \nu_n\}$ in such a way that  
$\mathbf{H}$ has $K_{\nu_i}$-rank at least two for $1\le i\le k$ and has $K_{\nu_i}$-rank $1$ 
for $k<i\le n$. The extreme cases $k=n$ or $k=0$ are, of course, possible. 

Let $X_i$ be the Bruhat-Tits building associated to $\mathbf{H}(K_{\nu_i})$, and 
let $H_i:= \mathbf{H}(K_{\nu_i})^+$. We write  
\[ 
	H=H_1\times\dots\times H_n~\text{ and }~X=X_1\times\cdots\times X_n.
\]
With that notation, $H=\tdf{G}_2$. 
The spaces $X_i$ are irreducible Euclidean buildings with cocompact
affine Weyl group. 
The group $H$ acts by automorphisms on the simplicial complex $X$,
and this action is \emph{strongly transitive} in the sense that $H$ acts transitively on pairs $(C,A)$
where $C\subset A$ is a chamber in an apartment.

The proof of Theorem~\ref{T:main} is completed by applying the following general theorem to the subgroup $H=\tdf{G}_2$ 
in $G=\tdf{G}_1$, using Remark~\ref{rem: tree extension up to finite index}.

\begin{theorem}\label{T:qi-step}
Let $X$ be a locally  finite affine building without Euclidean factors.
Let $H$ be a lcsc group acting isometrically on $X$ strongly transitively and with a compact kernel.
Assume $\phi\colon H\to G$ is a continuous homomorphism with compact kernel and a closed image of cofinite volume.\\
Then the action of $H$ on $X$ extends via $\phi$ to an isometric action with a compact kernel of $G$ on $X$.
\end{theorem} 

%
%

The proof of Theorem~\ref{T:qi-step} proceeds in several steps: 
reduction to the case where $G$ is tdlc (already given in our main application),
proof that the image of $H$ is cocompact in $G$, and finally the use of 
techniques from geometric group theory, similar to the approach taken in~\cite{Furman-MM}*{Section~3}. 

\subsubsection*{Reduction to $G$ being tdlc with trivial amenable radical} 
(This reduction is not needed for our application, where $G$ is $\tdf{G}_1$). 
We will first show that both the connected component $G^0$ and $\Radam(G)$ are compact normal subgroups,
so that dividing by them we may assume $G$ to be tdlc with the trivial amenable radical. 
A group acting strongly transitively on a thick Euclidean building has trivial amenable radical 
(cf. \cite{Caprace+Monod:I}*{Theorem 1.10}, where this is deduced from \cite{AdamsBallmann}).
Therefore $\Radam(H)$ is contained in $\ker(H\to\Aut(X))$, which is compact by assumption.
It follows from Lemma~\ref{L:Radam-and-fincovol} that $\Radam(G)\cap \phi(H)$ is a compact group.
Let $G'=G/\Radam(G)$ and let $H'$ be the image of $H$ in $G'$.
By \cite[Theorem~5.1]{BG} the image of $H/\Radam(H)$ in $G'$ is closed (note that $H/\Radam(H)$ is qss by 
\cite[Theorem~3.7]{BG}).
It follows that $M=H\cdot\Radam(G)$ is a closed subgroup of $G$, and applying Lemma~\ref{L:intermediate} to the closed subgroups
$H<M<G$ we conclude that $M/H\cong \Radam(G)/\Radam(H)$ has finite volume.
Since $\Radam(H)$ is compact, it follows that $\Radam(G)$ has finite Haar measure, and therefore is a compact group.

Dividing by the compact amenable radical, hereafter we assume $G$ and $H$ to have trivial amenable radicals. 
The connected component $G^0$ of the identity in $G$ is a connected, center free, semi-simple real Lie group, and we consider
the action of $H$ on $G^0$ by conjugation, providing a homomorphism $\rho:H\to \Aut(G^0)$.
Note that $\Aut(G^0)$ is a Lie group, and the image of $H$ is closed (cf. \cite[Theorem~5.1]{BG}).
Since $H$ is tdlc, the image $\rho(H)$ is discrete. 
Considering the isometric continuous action of $H$ on the discrete space $\rho(H)$ we conclude
by \cite[Theorem~6.1]{BG} that $H_0=\ker \rho$ has finite index in $H$.
As before the image of $H$ in $G/G^0$ is closed, so $H \cdot G^0$ is closed in $G$, and the same applies to
the finite index subgroup $H_0\cdot G^0$. 
Lemma~\ref{L:intermediate} applied to $H_0<H_0\cdot G^0<G$ shows that $H_0\cdot G^0/H_0\cong G^0/(H_0\cap G^0)$ has finite Haar measure.
But $H_0\cap G^0=\{1\}$ because $G^0$ is center-free. 
We conclude that $G^0$ has finite Haar measure, and therefore is a compact group.
Thus, for the rest of the proof we assume $G$ to be tdlc with trivial amenable radical.

\subsubsection*{Compactness of $G/H$} 
By the Bruhat--Tits fixed-point theorem (see~\cite{bridson+haefliger}*{Corollary~2.8 on p.~179} for a more general result on CAT(0)-spaces) every compact subgroup of $H$ fixes a vertex of $X$. But since there are only finitely many $H$-orbits of vertices there are only finitely many vertex stabilizer groups up to conjugation. 
In particular, $H$ endowed with Haar measure $m$ 
has an upper bound on the Haar measure of its open compact subgroups: 
\[
		\sup \bigl\{ m(U) \mid U\ \textrm{is\ an\ open\ compact\ subgroup\ of}\ H\bigr\}<+\infty.
\] 
Thus Lemma~\ref{L:bounded-clopen} applies to $H<G$ and implies compactness of $G/H$.

\medskip

Therefore we are in a position to apply Lemma~\ref{lem: QI of lcsc group} in combination with 
Remark~\ref{rem: continuous local cross section} to $H$ acting on $X$, 
to obtain a homomorphism of $G$ to the quasi-isometry group of~$X$, $\phi\colon G \rightarrow\ \QI(X)$.
In fact, we have a commutative diagram 
\[
	\begin{tikzcd}
      	H \ar[r]\ar[d] & \Isom(X)\ar[d, hook] \\
      	G \ar[r, "\phi"] & \QI(X).   
  	\end{tikzcd}
\]
Further, there are constants $C,L>0$ such that for every $g\in G$ there 
is a $(L,C)$-quasi-isometry $f_g\colon X\to X$ that represents the class $\phi(g)\in \QI(X)$ and has the following property: 
For every bounded set $B\subset X$ there is a neighborhood of the identity $V\subset G$ so that 
\begin{equation}\label{eq: finite distance to id}
			\forall_{g\in V}\forall_{x\in B}~~d_X(f_g(x),x)<C.
\end{equation}

In what follows we denote by $\bar{G}$ and $\bar{H}$ the corresponding images of $G$ and $H$ in $\QI(X)$;
the commutativity of the diagram above implies that $\bar{H}<\bar{G}$.
For a general metric space the QI-group does not have a natural topological group structure. 
In the arguments below we will take advantage of the large scale geometry of $X$ to obtains information
on $\bar{G}$ to be able to place it in the image of $\Isom(X)$ in $\QI(X)$.

\subsubsection*{Definition of the subgroup $\bar{G}'<\bar{G}$} 
We now apply the following splitting theorem by Kleiner and Leeb. 
\begin{theorem}[\cite{Kleiner+Leeb}*{Theorem 1.1.2}]\label{thm: splitting kleiner leeb}
For every $C,L>0$ there are $L', C', D>0$ such that 
every $(L,C)$-quasi-isometry $X\to X$ is within distance $D$ from
a product of $(L',C')$-quasi-isometries $X_i\to X_{\pi(i)}$ between the factors 
for some permutation $\pi\in\operatorname{Sym}_n$. 
\end{theorem} 

At the expense of increasing the constants $C$ and $L$, we may hence assume that for each $g\in G$, the quasi-isometry  
$f_g:X\to X$ as in (\ref{eq: finite distance to id}) is a product of $(L, C)$-quasi-isometries $f^{(i)}_g\colon X_i\to X_{\pi_g(i)}$ for 
some permutation $\pi_g\in {\rm Sym}_n$. 
Another consequence of Theorem~\ref{thm: splitting kleiner leeb} is 
that the product $\QI(X_1)\times\cdots\times\QI(X_n)$ embeds into $\QI(X)$ as a subgroup of index at most~$n!=|\operatorname{Sym}_n|$. 
We define the finite index subgroup $\bar{G}'<\bar{G}$ by
\[
	\bar{G}'=\bar{G}\cap \QI(X_1)\times\cdots\times\QI(X_n).
\] 
Let $\bar{H}'=\bar{H}\cap \bar{G}'$, and let $G'<G$ and $H'<H$ be the preimages of $\bar{G}'$ and $\bar{H}'$.
	

It follows that there are $(L,C)$-quasi-isometries $f_g^{(i)}\colon X_i\to X_i$ for every $i\in\{1,\dots, n\}$ 
and $g\in G'$ such that $f_g^{(i)}$ represents the $i$-th component of the quasi-isometry $f_g$.
In view of~\eqref{eq: finite distance to id} and the quantitative statement in Theorem~\ref{thm: splitting kleiner leeb} 
and at the expense of increasing $C, L>0$ once more, the following holds true: 
For every $i\in\{1,\dots,n\}$ and for every bounded set $B_i\subset X_i$ there is a neighborhood 
of the identity $V\subset G$ so that 
\begin{equation}\label{eq: finite distance to id - for the factors}
		\forall_{g\in V\cap G'}\forall_{x\in B_i}~~d_X(f_g^{(i)}(x),x)<C.
\end{equation}

\medskip

\subsubsection*{Openness of $G'<G$} 
The group $G'$ is defined as the kernel of a homomorphism of $G$ to a finite group. 
But we do not know at this point that the homomorphism is continuous, which would imply that $G'<G$ is open, thus an lcsc group itself. 
Next we provide a direct argument that shows openness of $G'<G$. 

Let $d_{X_i}$ denote the metric $X_i$, and the metric $d_X$ on $X$ being the $\ell^2$-sum of $d_{X_i}$-s.
Let $B_i\subset X_i$ be a bounded subset whose diameter exceeds $3C$, 
and let $B=B_1\times \dots\times B_n$. 
Let $V\subset G$ be a neighborhood of the identity that satisfies~\eqref{eq: finite distance to id}. 
Next we show that $V$ is contained in $G'$ and so $G'<G$ is an open subgroup. 
Suppose it is not. 
Then there is $g\in V$ such that $\pi_g$ is non-trivial, namely
there is $i\in\{1,\dots, n\}$ such that $j:=\pi_g(i)\ne i$. 
Let $x_j,x_j'\in B_j$ be points whose distance is at least $3C$. 
Pick points $x_l\in B_l$ for $l\in\{1,\dots, j-1, j+1,\dots,n\}$. 
Then either $d_{X_j}(f_g^{(i)}(x_i), x_j)>C$ or $d_{X_j}(f_g^{(i)}(x_i), x_j')>C$. 
Without loss of generality assume the first case. Let $x=(x_1,\dots, x_n)\in B$. 
Then 
\[ d_X(f_g(x), x)\ge d_{X_j}(f_g^{(i)}(x_i), x_j)>C,\]
contradicting~\eqref{eq: finite distance to id}. Therefore $V\subset G'$. 

\subsubsection*{Mapping $G'$ to $\Isom(X_i)$ for the higher rank factors $X_i$}
The rigidity theorem for 
higher rank irreducible buildings (such as $X_1,\dots,X_k$) by Kleiner--Leeb~\cite{Kleiner+Leeb} 
is the next important ingredient. 

\begin{theorem}[\cite{Kleiner+Leeb}*{Theorem~ 1.1.3}]\label{thm: rigidity kleiner leeb}
Let $i\in\{1,\dots, k\}$. 
For every $C>0$ and $L>0$ there is a constant $D>0$ such that every $(L,C)$-quasi-isometry $X_i\to X_i$ is within distance $D$ 
from a unique isometry $X_i\to X_i$. 
\end{theorem}

	Moreover, no two distinct isometries $X_i\to X_i$ for $i\in\{1,\dots,k\}$ are within bounded distance from each other. 
	Hence the natural homomorphism $\Isom(X_i)\to \QI(X_i)$ is an isomorphism 
	for $i\in\{1,\dots,k\}$ and we obtain 
	homomorphisms
	\[
		\psi_i\colon G'\longrightarrow \Isom(X_i)	
	\] 
	for which the restriction to $H'<G'$ is the homomorphism $H'\to\Isom(X_i)$. 
	The statement~\eqref{eq: finite distance to id - for the factors} and Theorem~\ref{thm: rigidity kleiner leeb} applied 
	to $f_g^{(i)}$ show that there is a constant $E>0$ such that 
	for every bounded subset $B\subset X_i$ there is a neighborhood of the identity 
	$V\subset G$ such that 
	\begin{equation}\label{eq: psi to id}
		\forall_{g\in V}\forall_{x\in B}~~d_{X_i}(\psi_i(g)(x),x)<E.
	\end{equation}

	\begin{lemma}\label{lem: cont-rank2}
		The map $\psi_i$ is continuous for $i\in\{1,\dots, k\}$. 
	\end{lemma}
	\begin{proof}
	Let $U\subset \Isom(X_i)$ be an open neighborhood of the identity. We have to 
	show that there is an open neighborhood of the identity of $G$ that is contained 
	in $\psi_i^{-1}(U)$. 
	
	We rely on the following geometric fact about buildings that follows, for example, from~\cite{Ramos-Cuevas}: 
	For every constant $D>0$ and any open neighborhood $W$ of the identity in $\Isom(X_i)$,
	there is a bounded set $B\subset X_i$, depending on $D$ and $W$ so that 
	\begin{equation}\label{e:geo-cpx 1}
		\bigl\{ \theta\in \Isom(X_i) \mid \sup_{x\in B}\ d_{X_i}(\theta(x),x)<D\bigr\}\subset W.
	\end{equation}
	We apply this general fact to the constant $E$ from~\eqref{eq: psi to id} and 
	the identity neighborhood $U$. Let us fix a bounded subset $B=B(E,U)\subset X_i$ such that 
	\begin{equation}\label{eq: dist for isometries 1}
		 \bigl\{\theta\in \Isom(X_i) \mid \sup_{x\in B}\ d_{X_i}(\theta(x),x)<E\bigr\}\subset U.
	\end{equation}
    Applying the statement~\eqref{eq: psi to id} to this specific subset $B$ provides us with a neighborhood $V=V(E,B)\subset G$ 
	of the identity such that 
	\begin{equation}\label{eq: dist psi 1}
			\forall_{g\in V}\forall_{x\in B}~~d_{X_i}(\psi_i(g)(x),x)<E. 
	\end{equation}
	Since~\eqref{eq: dist for isometries 1} and~\eqref{eq: dist psi 1} 
	mean that $V\subset \psi_i^{-1}(U)$, continuity of $\psi_i$ follows. 
	\end{proof}
	
	\medskip
	
\subsubsection*{Mapping $G'$ to $\Homeo(\partial X_j)$ for the rank~$1$ factors $X_j$}
	Next we turn to the tree factors $X_j$ with $j\in\{k+1,\dots,n\}$. 
	For every tree of bounded degree $T$, such as $X_j$ with $k< j\le n$, the space of ends $\partial T$ is compact, 
	and one has an embedding of groups 
	\[
		\QI(T)\hookrightarrow\Homeo(\partial T).
	\]
	For $j\in \{k+1, \dots, n\}$ let $\psi_j$ be the composition  
	\[
		\psi_j\colon G'\ \xrightarrow{\phi_j}\ \QI(X_j)\ \hookrightarrow\ \Homeo(\partial X_j)
	\]
	and will denote by $\bar{G}_j'=\psi_j(G')$ the image in $\Homeo(\partial X_j)$.
	
	\begin{lemma}\label{lem: cont-rank1} 
		The map $\psi_j$ is continuous and has a closed image for every 
		$j\in \{k+1, \dots, n\}$. 
	\end{lemma}
	\begin{proof}
	    The continuity statement is proved in~\cite[Theorem 3.5]{Furman-MM}. 
		The argument is analogous to the one for Lemma~\ref{lem: cont-rank2}. 
		However, instead of the geometric fact about higher rank buildings \eqref{e:geo-cpx 1} 
		one uses the following fact about trees; it is a consequence of the Mostow--Morse lemma for Gromov hyperbolic spaces: 
			
		For a tree $T$ of bounded degree and constants $L,C,D>0$ and for any identity neighbourhood $U\subset \Homeo(\partial T)$ 
		there is a compact subset $B\subset T$ so that 
		the image in $\Homeo(\partial T)$ of any $(L,C)$-quasi-isometry $f\colon T\to T$ with  
			\[\sup\{d_T(f(x),x) \mid x\in B\}<D\]
			belongs to $U$. 
			
		Since the image of $H'$ is cocompact in $G'$ and the image of $H'$ in $\Homeo(\partial X_j)$, 
		is closed, the image $\bar{G}_j'$ of $G'$ in $\Homeo(\partial X_j)$ is also closed. 
\end{proof}
	
\medskip
	
\subsubsection*{Identifying the image of $\psi_j$}
	
	%
	%
	%
	%
	%

	Fix an index $j\in\{k+1,\dots, n\}$.
	Let $H'_j$ denote the image of $H'$ in $\isom(X_j)$, and $\bar{H}'_j$ denote the corresponding 
	subgroup of $\Homeo(\partial X_j)$. 
	Consider the action of $\bar{G}'_j$ on the Cantor set $\partial X_j$,
	that comes from the quasi-action of $\bar{G}'_j$ on the tree $X_j$.
	The latter is given by the family $\{f_g^{(j)}\}$ of $(L,C)$-quasi-isometries of the locally finite tree $X_j$,
	such that $f_{g_1}^{(j)}\circ f_{g_2}^{(j)}$ 
	is within uniformly bounded distance from $f_{g_1g_2}^{(j)}$ for all $g_1,g_2\in G'$.
	Note that the maps $f^{(j)}_g$ may be associated to the images $\psi_j(g)\in \bar{G}'_j$ rather than $g\in G'$.
	We proceed with two claims:
	\begin{enumerate}
		\item 
		There exists a locally finite tree $T_j$, an action $\rho\colon\bar{G}'_j\to \isom(T_j)$,
		and a quasi-isometry $q\colon T_j\to X_j$, so that $f^{(j)}_{g}$ are within bounded distance from the map $q\circ \rho(g)\circ q'$
		where $q'\colon X_j\to T_j$ is a coarse inverse of $q$.
		\item 
		One can take tree $T_j$ to be $X_j$.
	\end{enumerate}
	Claim (1) follows from the deep work of Mosher--Sageev--Whyte on quasi-actions on trees \cite{MSW-annals}*{Theorem~1}.
	We would like to thank the anonymous referee for pointing out an alternative argument based on the result of Carette--Dressen \cite{CD}*{Theorem~C}.
	Indeed, the fact that the action of $H'_j$ on the tree $X_j$ is strongly transitive,
	implies that the action of $\bar{H}'_j$ on the Cantor set $\partial X_j$ 
	is \emph{$3$-proper} and \emph{$3$-cocompact} (meaning the action on the space of distinct triples is proper and cocompact), 
	and these properties pass to the action of
	$\bar{G}'_j$ that contains $\bar{H}'_j$ as a cocompact subgroup.
	By \cite{CD}*{Theorem~C} such an action of a locally compact group on a Cantor set is sufficient to construct an isometric action
	on a locally finite tree $T_j$ with $\partial T_j$ being identified with the given Cantor set.

	A few general remarks on trees and groups acting on trees.
	An action of a group on a tree is called \emph{minimal} if there are no proper subtrees invariant under the group action. 
	If an action on a locally finite tree is minimal then the induced action on the boundary of the tree is minimal in the sense of dynamical systems:
	there are no proper closed invariant subsets. 
	This is a consequence of the limit set being the unique minimal closed invariant subset of the boundary 
	(see~\cite{coornaaert}*{Th\'eor\`eme~5.1} for a more general result, which is attributed to Gromov, in the context of Gromov-hyperbolic spaces). 
	
	Conversely, if a group $G$ is acting by automorphisms on a locally finite tree $T$, and $L\subset \partial T$ is a minimal $G$-invariant set
	consisting of more than two points, then the union of all geodesics in $T$ with end points in $L$, called the \emph{convex hull} $\operatorname{co}(L)$ of $L$, 
	is a minimal $G$-invariant sub-tree of $T$. 
	Any $G$-invariant sub-tree $T'$ has $\partial T'=L$, and therefore contains $\operatorname{co}(L)$. 
	Moreover there exists a retraction $r\colon T'\to \operatorname{co}(L)$ given by nearest point projection to $\operatorname{co}(L)$.
	This retraction is equivariant with respect to $\Aut(T)$.
	
	Returning to our situation, we may assume  without loss of generality that the tree $T_j$ has no vertices of degree one
	and so $T_j=\operatorname{co}(\partial T_j)$.
	We identify its boundary $\partial T_j$ with $\partial X_j$ in a $\bar{G}'_j$-equivariant way. 
	Since the action of $H_j$ on the tree $X_j$ is minimal, the topological action of $\bar{H}'_j$ on the Cantor set $\partial X_j=\partial T_j$ is minimal,
	and therefore the action of $\bar{H}'_j$ on the tree $T_j$ is minimal as well, i.e. $T_j$ does not contain any proper $\bar{H}'_j$-invariant subtree.

	\begin{lemma}\label{lem: action map continuous and injective}
		The action homomorphism $\bar{G}'_j\to \isom(T_j)$ is injective, continuous and has a closed image. 
	\end{lemma}
	\begin{proof}	
		Since both the isometry group and the quasi-isometry group of either $X_j$ or $T_j$ inject into the homeomorphism group of $\partial X_j=\partial T_j$, 
		the map $\bar{G}'_j\to \isom(T_j)$ is injective. 
		Since $\isom(T_j)$ embeds as a closed subgroup into $\Homeo(\partial T_j)$ continuity and closedness follow 
		from the continuity and closedness of the composition 
		\[
			\bar{G}'_j\to\isom(T_j)\to \Homeo(\partial T_j)=\Homeo(\partial X_j).\qedhere
		\] 
	\end{proof}

	The argument for claim (2) relies on the fact that the structure of the tree $X_j$
	can be read of the intersection patterns of vertex stabilizers, which are maximal open compact subgroups of $\bar{H}'_j$
	(see e.g.~\cite{abramenko+brown}*{Corollary 11.35 on p.~563 and Theorem~11.38 on p.~564} for the  
	more general discussion of groups with BN-pairs acting on Euclidean buildings).
	In our case, $\bar{H}'_j$ acts on the tree $X_j$, transitively on edges and with two 
	orbits of vertices. Let $e=(v_1,v_2)$ be an edge in $X_j$, and denote 
	\[
		K_i=\Stab_{\bar{H}'_j}\left(\{v_i\}\right)\qquad (i=1,2).
	\]
	Then $K_1$ and $K_2$ are maximal compact (and open) non-conjugate subgroups of $\bar{H}'_j$. 
	The group $\bar{H}'_j$ is generated by $K_1\cup K_2$, and it acts transitively on the edges of $X_j$,
	with $K_1\cap K_2$ being the stabilizer of edge $e=(v_1,v_2)$.
	For each $i=1,2$, the subgroup $K_1\cap K_2$ is a maximal subgroup in $K_i$ and all maximal subgroups on $K_i$ are conjugate in $K_i$
	(because $K_i$ acts transitively on the edges emerging from $v_i$).

	\begin{lemma}\label{lem: homothety between trees} 
		There is $k\in\bbN$ and an $\bar{H}'_j$-equivariant  cellular homeomorphism $f\colon X_j\to T_j$ 
		such that $f$ is a homothety with stretch factor $k$. 	
	\end{lemma}
	\begin{proof}
		Upon subdividing we may assume that $\bar{H}'_j$ acts on $T_j$ without inversion. 
	    By the continuity part of Lemma~\ref{lem: action map continuous and injective} the orbits of the compact 
		subgroups $K_1$ and $K_2$ in $T_j$ are bounded. 
		By the Bruhat--Tits fixed point theorem~\cite{bridson+haefliger}*{Corollary~2.8 on p.~179}, there is a  vertex in $T_j$ that is  
		fixed under $K_1$ or $K_2$, respectively. Choose a pair $(w_1, w_2)$ of vertices in $T_j$ with $K_iw_i=w_i$ that has minimal distance. 

		Since $\bar{H}'_j$ is generated by $K_1$ and $K_2$, the $\bar{H}'_j$-orbit 
		$T'=\bigcup_{h\in \bar{H}'_j} [hw_1,hw_2]$ of the geodesic segment $[w_1,w_2]$ 
		between $w_1$ and $w_2$ is connected.
		Thus $T'$ is a $\bar{H}'_j$-invariant subtree of $T_j$.
		By the proceeding discussion this implies $T'=T_j$.
		
		The segment $[w_1,w_2]$ is fixed by $K_1\cap K_2$. 
		The map that sends the vertex $v_i$ to $w_i$ for $i=1,2$, 
		and the edge $[v_1,v_2]$ onto $[w_1,w_2]$ by an affine homeomorphism, can be extended to an $\bar{H}'_j$-equivariant 
		cellular surjective map $f\colon X_j\to T_j$. 

		Next we show that $f$ is locally injective. It is enough to show local injectivity at the $K_i$-fixed vertex $v_i$ of $X_j$. 
		By symmetry let us consider $v_1$ only. 
		Consider its neighbours $v_2$ and $v'_2=hv_2$ for some $h\in \bar{H}'_j$. 
		Since $K_1$ and $K_2$ are not conjugated, the $f$-images $w_1, w_2, w_2'$ of $v_1$, $v_2$ and $v'_2$ are pairwise distinct. 
		Let $w$ be the center of the tripod given by $w_1, w_2, w'_2$. 
		We want to show that $w=w_1$ which implies local injectivity of $f$ at $v_1$. 
		Both $K_1\cap K_2$ and $K_1\cap hK_2h^{-1}$ are contained in the stabilizer $S$ of $w$. 
		Since $K_1\cap K_2$ is maximal in $K_1$ and $K_1\cap hK_2h^{-1}\ne K_1\cap K_2$ 
		it follows that $S=K_1$. 
		So $w=w_1$ since we chose $(w_1,w_2)$ to have minimal 
		distance among the vertices fixed by $K_1$ or $K_2$, respectively. 

		As a surjective, locally injective map between trees $f$ is a homeomorphism. 
		Let $k$ be the distance between $w_1$ and $w_2$. 
		Obviously, $f$ is locally a homothety with stretch factor $k$. 
		Thus it is globally so. 	
	\end{proof}

    Conjugating with a homothety as in Lemma~\ref{lem: homothety between trees} we obtain a topological isomorphism 
    \[
		\isom(T_j)\cong \isom(X_j)
	\]
    compatible with the two embeddings of $\bar{H}'_j$. 
	Hence, $\bar{G}'_j$ can be embedded as a closed intermediate subgroup $\bar{H}'_j<\bar{G}'_j<\Isom(X_j)$, 
	and the homomorphism $\psi_j$ can be regarded as a homomorphism $G'\to \bar{G}'_j\hookrightarrow \Isom(X_j)$ composed 
	with the natural embedding $\Isom(X_j)\to\Homeo(\partial X_j)$. 
	Thus we also denote the map $G'\to \bar{G}'_j\to \Isom(X_j)$ by $\psi_j$. 
	Let $\Delta\colon G'\to \prod_{i=1}^n G'_i$ be the diagonal embedding. 
	We obtain a continuous homomorphism with closed image 
	\[
		\psi'\colon G'\xrightarrow{\prod \psi_i\circ\Delta} \prod_{i=1}^n \Isom(X_i)
	\]
	whose restriction to $H'$ corresponds to the embedding $\bar{H}'\hookrightarrow \prod_{i=1}^n \Isom(X_i)$. 
	
	Finally we argue that this homomorphism $\psi':G'\to \prod \bar{G}_i'\hookrightarrow \prod \Isom(X_i)$ extends to 
	a homomorphism $\psi\colon G\to \Isom(\prod X_i)$.
	Recall that by Theorem~\ref{thm: splitting kleiner leeb}, the image $\bar{G}$ of $G$ in $\QI(X)$ where $X=X_1\times \cdots\times X_n$,
	is a finite extension of the product $\bar{G}'=\bar{G}'_1\times \cdots \times \bar{G}'_n$ with $\bar{G}'_i<\QI(X_i)$,
	by the finite group $\pi(G)<{\rm Sym}_n$.
	The conjugation action of $G$ on $G'=G'_1\times \cdots \times G'_n$ that descends to the conjugation action 
	of $\bar{G}$ on $\bar{G}'=\bar{G}'_1\times \cdots\times \bar{G}'_n$, and for each $g\in G$
	induces isomorphisms $p_{g,i}:\bar{G}'_i\to \bar{G}'_{\pi_g(i)}$.
	As a consequence of the proof so far we identify the groups $\bar{G}'_i$ as closed intermediate subgroups 
	\[
		\bar{H}'_i<\bar{G}'_i<\Isom(X_i)\qquad (i\in \{1,\dots, n\}) 
	\]
	and observe that the isomorphisms $p_{g,i}:\bar{G}'_i\to \bar{G}'_{\pi_g(i)}$ are continuous
	(as in Lemmas~\ref{lem: cont-rank2} and \ref{lem: action map continuous and injective}).
	Any continuous group isomorphism $\bar{G}'_i\cong \bar{G}'_j$
	is induced by an isometry $X_i\cong X_{j}$ -- for higher rank buildings (corresponding to $i,j\le k$) 
	this follows from Theorem~\ref{thm: rigidity kleiner leeb} of Kleiner--Leeb, 
	and for rank one cases (corresponding to $k\le i,j\le n$) it can be deduced from Lemma~\ref{lem: homothety between trees}.
	Therefore, for every $g\in G$, $\pi_g(i)=j$ only if $X_i$ is isometric to $X_j$.
	Thus each $g\in G$ defines an isometry $\psi(g)$ of $X=X_1\times \cdots\times X_n$,
	thereby defining the claimed continuous homomorphism $\psi:G\to \isom(\prod_{i=1}^n X_i)$ that extends
	$\psi':G'\to \prod \Isom(X_i)$.

	This completes the proof of Theorem~\ref{T:qi-step}, and therefore also the proof of the main classification
	result -- Theorem~\ref{T:main}.

\section{Proofs of Theorems~\ref{T:Mostow}--\ref{T: characterisation of groups with positive first L2 Betti}}\label{sec: applications}

\begin{proof}[Proof of Theorem~\ref{T:Mostow}]\hfill{}\\
Let $\Gamma<H$ be a lattice embedding that is 
\begin{itemize}
	\item[(i)] 
		either an irreducible lattice embedding into a connected, center-free, semi-simple real Lie group~$H$ without compact factors, 
		which is not locally isomorphic to $\SL_2(\bbR)$, 
	\item[(ii)] 
		or an $S$-arithmetic lattice embedding as in Definition~\ref{def: definition of S-arithmetic lattice up to tree extension}. 
\end{itemize}

Such $\Gamma$ satisfies all the assumptions of Theorem~\ref{T:main}.
Being a lattice, the group $\Gamma$ is Zariski dense in the semi-simple algebraic group $H$ 
by Borel's Density Theorem~\cite{Margulis:book}*{Theorem II.(4.4)}. 
Thus Theorem~\ref{T:good-groups}.(2) ensures that $\Gamma$ satisfies conditions $\CAF$ and $\NBC$.
Being an irreducible lattice, $\Gamma$ satisfies $\IRR$.
Moreover, lattices in semi-simple groups are known to be finitely generated~\cite{Margulis:book}*{\S IX.(3.2)}. 
So any lattice envelope of $\Gamma$ is compactly generated by Lemma~\ref{L:fg-cg}. 
We also remark that neither $\Gamma$ nor any finite index subgroup $\Gamma'<\Gamma$ have non-trivial finite normal subgroups.
Indeed, such a subgroup group $N\normal \Gamma'$ would be centralized by a finite index subgroup $\Gamma''=\ker(\Gamma'\to \Aut(N))$,
that forms an irreducible lattice in $H$.  
So $N\subset \calZ_H(\Gamma'')$, while by Borel's Density Theorem \cite{Margulis:book}*{\S II.(6.3)},  
$\calZ_H(\Gamma'')=\calZ(H)$ is the center of $H$, which is trivial.

Let $G$ be some lcsc group and $\Gamma < G$ be a lattice embedding. 
By Theorem~\ref{T:main} there is an open subgroup of finite index $G'<G$ and a compact
normal subgroup $K\normal G'$ so that, denoting 
$G'_0:=G'/K$ and $\Gamma':=\Gamma\cap G'$, the lattice embedding $\Gamma'\hookrightarrow G'_0$  
(by the above discussion $\Gamma'\cap K=\{1\}$) 
induced from $\Gamma<G$ is of one of the following three types: 
\begin{enumerate}
	\item 
	$G'_0$ is a connected, center-free, semi-simple, real Lie group without compact factors and 
	$\Gamma'\hookrightarrow G'_0$ is an irreducible lattice;
	\item 
	$\Gamma'\hookrightarrow G'_0$ is an $S$-arithmetic lattice, possibly up to tree extension, 
	in the sense of Definition~\ref{def: definition of S-arithmetic lattice up to tree extension};
	\item 
	$G'_0$ is a tdlc group with trivial amenable radical. 
\end{enumerate}
First, note that $\Gamma<H$ as in (i) is incompatible with $\Gamma'<G'_0$ being of type (2),
and $\Gamma<H$ as in (ii) is incompatible with $\Gamma'<G'_0$ being of type (1).
Indeed, Margulis' super-rigidity theorem \cite{Margulis:book}*{Theorem VII.(7.1)} 
precludes the same group $\Gamma'$ from being both an $S$-arithmetic lattice 
(with both real and non-archimedean factors unisotropic factors present) 
and being a lattice in a purely real Lie group.

We will show in Theorem~\ref{thm: discreteness in case 3} below that $\Gamma<H$ as in (i) or (ii) is incompatible with non-trivial 
lattice embeddings $\Gamma'<G'_0$ of type (3).
Before delving into this argument, let us now discuss in some detail the situations, where 
$\Gamma<H$ is as in (i) and $\Gamma'<G'_0$ is of type (1), 
and then the case of $\Gamma<H$ as in (ii) and $\Gamma'<G'_0$ of type (2).
We recall the famous Strong Ridigity Theorem. 
\begin{theorem}[Strong Rigidity Theorem, Mostow, Prasad, Margulis]\hfill{}\\
	Let $H$ and $H^\flat$ be connected, center-free, semi-simple real Lie groups without non-trivial compact factors, 
	with $H\not\cong \PSL_2(\bbR)$, and let 
	$\Lambda<H$ and $\Lambda^\flat<H^\flat$ be irreducible lattices. 
	Then any isomorphism $\Lambda\cong \Lambda^\flat$ extends to a continuous isomorphism $H\cong H^\flat$.
\end{theorem}
Strong Rigidity Theorem was discovered by Mostow \cite{Mostow:1973book} for all the cases where $H$, $H^\flat$  
are semi-simple real Lie groups and $\Lambda<H$, $\Lambda^\flat<H^\flat$ are \emph{uniform} irreducible lattices.
Prasad \cite{Prasad-Mostow} extended this result to some non-uniform lattices, including all non-uniform lattices in rank one real Lie groups.
Higher rank cases, for both uniform and non-uniform lattice embeddings, 
follow from Margulis' super-rigidity theorem (cf. \cite{Margulis:book}*{Theorem VII.(7.1)}).
The latter result is even more general, and implies Strong Rigidity for irreducible lattices in a larger class of groups,
including $\Lambda<H$ being an $S$-arithmetic lattice as in Setup~\ref{setup:arith} and $H^\flat$ being a product of 
real and non-archimedean simple Lie groups over local fields.

\medskip

Let $\Gamma<H\not\cong \PSL_2(\bbR)$ be an irreducible lattice in a real semi-simple group as in (i),
and assume that a finite index subgroup $\Gamma'<\Gamma$ and the group $G'_0$ are as in (1).
Then by Strong Rigidity Theorem, applied to $\Lambda=\Gamma'$ and $H^\flat=G'_0$, we get a continuous isomorphism $\rho':G'_0\cong H$
that intertwines the lattice embedding $j'\colon \Gamma'\hookrightarrow G'\rightarrow G'_0$ and the  
restriction $i'=i|_{\Gamma'}$ of the lattice embedding $i\colon\Gamma\hookrightarrow H$.

Peeking into the proof of Theorem~\ref{T:main} (see Theorem~\ref{T:lcsc-mod-ramen} and Proposition~\ref{P:step1}) 
we see that $K\normal G'$ is the amenable radical of $G$. 
So we also have the induced lattice embedding $\Gamma\hookrightarrow G_0:=G/K$. 
The action of $G$ on $G'$ by conjugation, gives a continuous homomorphism $\pi\colon G\to \Aut(H)$
with $K<\ker (\pi)$ so that the restriction $\pi'=\pi|_{G'}$
coincides with the composition $G'\to G'_0\overto{\rho'} H\cong \operatorname{inn}(H)\hookrightarrow \Aut(H)$.
This shows that the following diagram is commutative, i.e. that two homomorphisms $\pi'\circ j'$ and $\operatorname{inn}\circ i'$
describing embeddings $\Gamma'\hookrightarrow H\hookrightarrow \Aut(H)$, coincide:
\begin{equation}\label{e:Gamma-prime1}
	\begin{tikzcd}
  		\Gamma' \ar[d, hook, "j'"] \ar[rrr, hook, "i'"] &  & &H \ar[d, hook, "\operatorname{inn}" ]\\
  		G' \ar[r] & G'_0 \ar[r, "\rho'"] & H\ar[r,hook]&\Aut(H),   
	\end{tikzcd}
\end{equation}
We claim that the diagram below is also commutative, i.e. that the homomorphisms $\alpha=\pi\circ j:\Gamma\to \Aut(H)$ 
and $\beta=\operatorname{inn}\circ i:\Gamma\to \Aut(H)$ also coincide.
\begin{equation}\label{e:Gamma-2}
	\begin{tikzcd}
  		\Gamma \ar[d, hook, "j"] \ar[rr, hook, "i"] &  & H \ar[d, hook, "\operatorname{inn}" ]\\
  		G \ar[r]\ar[rr, bend right, "\pi"] & G_0 \ar[r] & \Aut(H).   
	\end{tikzcd}
\end{equation}
Indeed, $\Gamma'$ is normal in $\Gamma$ and $\alpha|_{\Gamma'}=\beta|_{\Gamma'}$, so for any $\gamma\in\Gamma$ and $\gamma'\in\Gamma'$
we have $\alpha(\gamma')=\beta(\gamma')$ and $\alpha(\gamma \gamma'\gamma^{-1})=\beta(\gamma \gamma'\gamma^{-1})$.
This implies that the element $\alpha(\gamma)^{-1}\beta(\gamma)\in \Aut(H)$ centralizes the lattice $\alpha(\Gamma')=\beta(\Gamma')$ in $H$.
Borel's Density Theorem and the fact that $\mathcal{Z}(H)=\{1\}$,
imply that $\alpha(\gamma)=\beta(\gamma)$ for every $\gamma\in\Gamma$.
This proves commutativity of the diagram (\ref{e:Gamma-2}),
and thereby proves Theorem~\ref{T:main} for $\Gamma<H$ as in (i) and $\Gamma'<G'_0$ satisfying (1).

\medskip

Next assume $\Gamma<H$ is an $S$-arithmetic lattice as in (ii) and that $\Gamma'<G'_0$ is of type (2),
i.e. is an $S$-arithmetic lattice, possibly up to a tree extension.
We view $G'_0$ as an intermediate closed group $H^\flat<G'_0<\Isom(X)$ where $H^\flat=\prod \mathbf{H}^\flat_i(K_{\nu_i})$
is a product of real and non-archimedean simple Lie groups, $X=\prod X_i$ is a product of irreducible symmetric spaces
and irreducible Bruhat--Tits buildings. 
In fact, we can write $X=\ssf{X}\times \tdf{X}$ where $\ssf{X}$ is a (possibly reducible) symmetric space 
and $\tdf{X}$ is a (possibly reducible) Euclidean building; the notation is motivated by the inclusions
$\ssf{(H^\flat)}<\Isom(\ssf{X})$ and $\tdf{(H^\flat)}<\Isom(\tdf{X})$.

Note that the image of $\Gamma'$ in $G'_0$ is contained in $H^\flat$, so there are two
$S$-arithmetic lattice embeddings $i'\colon\Gamma'\hookrightarrow H$ and $j'\colon\Gamma'\hookrightarrow H^\flat$.
The super-rigidity theorem (\cite{Margulis:book}*{Theorem VII.(7.1)}) implies a form of Strong Rigidity
that provides a continuous isomorphism $\pi'\colon H^\flat\to H$ intertwining these embeddings: $\pi'\circ j'=i'$,
similarly to the situation in diagram (\ref{e:Gamma-prime1}).
In particular, we may think of $H$ and $G'_0$ acting by isometries on the same space $X=\ssf{X}\times \tdf{X}$.

Note that the connected component $(G/K)^0$ of the identity is precisely $\Isom(\ssf{X})^0$,
and the action $G$ by automorphisms of $(G/K)^0$ gives a continuous homomorphism $G\to \Isom(\ssf{X})$.
Thus the continuous homomorphism $G'\to G'_0\hookrightarrow \Isom(\ssf{X})^0\times \Isom(\tdf{X})$ 
extends to the continuous homomorphism  
\[
	G\to  G_0 \hookrightarrow \Isom(\ssf{X})\times \Isom(\tdf{X}) <\Isom(X).
\]
Considering two isometric actions of $\Gamma$ on $X$ that agree on the normal subgroup of finite index $\Gamma'$,
and using the fact that the centralizer of $\Gamma'$ in $\Isom(X)$ is trivial, we deduce an analogue of 
diagram (\ref{e:Gamma-2}) is also commutative, i.e. that the two $\Gamma$-actions on $X$ coincide. 
This completes the proof of Theorem~\ref{T:main} in case (ii) and (2). 

\medskip

The following statement will complete the proof of Theorem~\ref{T:Mostow}. 
\begin{theorem}\label{thm: discreteness in case 3}
	Let $\Gamma<H$ be a lattice embedding as in (i) or (ii). Then any lattice embedding of $\Gamma$ into a tdlc group with trivial amenable radical is trivial. 
\end{theorem}

Let us say that a Polish group $L$ has property NSS (No Small Subgroups), if 
there is an identity neighborhood $1\in V\subset L$ that contains no
non-trivial compact subgroups. 
Examples of such groups include discrete countable groups, $\Homeo(S^1)$ -- the group of homeomorphisms of the circle
with compact open topology, and all real Lie groups. 
In proving Theorem~\ref{thm: discreteness in case 3} we will use the following.
\begin{lemma}\label{L:NSS}
	Let $\Gamma<G$ be a lattice embedding in a tdlc group $G$ with $\Radam(G)=\{1\}$.
	Let $L$ be a Polish group property NSS and  $\rho\colon G\to L$ be a continuous homomorphism that is injective on $\Gamma$.
	Then $\Gamma<G$ is a trivial lattice embedding.
\end{lemma}
\begin{proof}
    Let $V\subset L$ be a neighborhood that does not contain non-trivial compact subgroups. 
	In a tdlc group compact open subgroups form a basis of neighborhoods of the identity.
	Let $U$ be a compact open subgroup of $G$ contained in $\rho^{-1}(V)$. 
	Then $\rho(U)$ is a compact subgroup of $V$, therefore it is trivial. 
	Hence the subgroup $\ker(\rho)$ is open since, as it contains an open subgroup $U$.
	As $\ker(\rho)\cap \Gamma=\{1\}$ and $\Gamma<G$ is a lattice, $\ker(\rho)$ has finite Haar measure.
	Thus $\ker(\rho)$ is compact, and is therefore trivial because $\Radam(G)=\{1\}$ by assumption. 
	This implies that $G$ is discrete, and is a trivial lattice envelope for $\Gamma$.
\end{proof}

By Selberg's lemma $\Gamma$ as above is virtually torsion-free. 
Hence any lattice embedding of $\Gamma$ into a tdlc group is uniform according to Lemma~\ref{L:bounded-clopen}. 
To construct the homomorphism $\rho:G\to L$ as in Lemma~\ref{L:NSS} we shall use the homomorphism
\begin{equation}\label{e:rho0}
	\rho_0\colon G \to  \QI(G)\xrightarrow{\cong}\QI(\Gamma).
\end{equation}
We consider several cases.

\subsection*{Case $\Gamma<H$ is a non-uniform lattice}
By a theorem of Struble~\cite{struble} the group $G$ possesses a continuous, proper, left-invariant metric~$d$, 
and $d$ has to be quasi-isometric to the word-metric. 
So the kernel can be described as 
\[
	\ker(\rho_0)=\{g\in G\mid \sup_{h\in G} d(gh, h)<\infty\}
\]
which implies that $\ker(\rho_0)$ is a Borel subgroup of $G$. 	

For non-uniform, irreducible lattice $\Gamma<H$ the group $\QI(\Gamma)\cong \QI(G)$ coincides with 
the \emph{commensurator} of $\Gamma$ in $H$, which is a countable group. 
This was first proved by Schwartz~\cite{Schwartz} for rank-one real Lie groups, 
and then by Schwartz \cite{Schwartz2}, Farb--Schwarz~\cite{Farb+Schwartz}, and Eskin \cite{Eskin} 
for higher rank real Lie groups, and by Wortman~\cite{Wortman} for $S$-arithmetic cases.

Since $\ker(\rho_0)$ has countable index in $G$, the Borel subgroup $\ker(\rho_0)$ has positive Haar measure. 
Thus $\ker(\rho_0)$ is an open subgroup~\cite{bourbaki}*{IX \S 6 No 8, Lemma~9}.
Thus the quotient homomorphism 
\[
	G\overto{} L:=G/\ker(\rho_0)
\] 
is a continuous homomorphism into a countable discrete group. 
This homomorphism is injective on $\Gamma$ because $\Gamma$ embeds in $\QI(\Gamma)$.
Thus the proof of Theorem~\ref{thm: discreteness in case 3} is concluded by applying Lemma~\ref{L:NSS}. 

\medskip

For the rest of the proof we assume $\Gamma<H$ to be a uniform lattice,
write $H$ as a product of simple factors  $H=H_1\times\cdots \times H_n$ and let 
$X=X_1\times \cdots X_n$ be the product of the associated symmetric spaces or Bruhat--Tits buildings. 
At least one of the factors $H_i$ is a simple real Lie group, and we shall assume $H_1$ is such.
Since $\Gamma$ is quasi-isometric to $X$, we can view the homomorphism (\ref{e:rho0}) as
\[
	\rho_0:G\overto{} \QI(\Gamma)\cong \QI(X).
\] 
In view of the splitting Theorem~\ref{thm: splitting kleiner leeb} of Kleiner--Leeb ,
there is an open normal subgroup of finite index $G'<G$ whose image $\rho_0(G')$
is contained in the subgroup $\QI(X_1)\times \cdots \times \QI(X_n)$ of finite index in $\QI(X)$.
Denoting $\Gamma'=\Gamma\cap G'$ the homomorphism $\rho_0$ 
is compatible with the natural embedding $\Gamma'\to \Isom(X_1)\to \QI(X_1)$:
\[
	\begin{tikzcd}
  		\Gamma' \ar[d, hook] \ar[r, hook] & \Isom(X_1)\times \cdots\times \Isom(X_n) \ar[r, "\pr_1"] & \Isom(X_1) \ar[d, hook ]\\
  		G' \ar[r] & \QI(X_1) \times \cdots\times \QI(X_n) \ar[r, "\pr_1"] & \QI(X_1).   
	\end{tikzcd}
\]
Let us now consider the homomorphism $G'\to \QI(X_1)$, and denote by $\bar{G}'_1$ its image.

\subsection*{Case $\Gamma<H$ is uniform and $X_1=\HYP^2$}
(This case occurs if $H_1\cong \PSL_2(\bbR)$ is one of several factors of $H$). 
Note that the projection $\Gamma\to H_1$ is injective.
The boundary $\partial \HYP^2$ is homeomorphic to the circle $S^1$ and we have a homomorphism
\[
	\rho\colon G'\to  \QI(\HYP^2)\hookrightarrow \Homeo(S^1).
\]
This homomorphism is continuous by \cite{Furman-MM}*{Theorem~3.5}, and the restriction to $\Gamma$ is the composition of injective homomorphisms 
$\Gamma\overto{} H_1\cong\PSL_2(\bbR)\hookrightarrow \Homeo(S^1)$.
The Polish group $\Homeo(S^1)$ has NSS property, so Lemma~\ref{L:NSS} applies, and 
the proof of Theorem~\ref{thm: discreteness in case 3} is concluded in this case.

\subsection*{Case $\Gamma<H$ is uniform and $X_1=\HYP^m$ with $m\ge 3$.}
%
The boundary at infinity $\partial \HYP^m$ is a sphere $S^{m-1}$ with a natural conformal structure that 
is preserved by $\Isom(\HYP^m)$. 
Moreover there is an isomorphism $\Isom(\HYP^m)\cong\operatorname{Conf}(S^{m-1})$ that
extends to an isomorphism 
\[
	\QI(\HYP^m)\cong \operatorname{QConf}(S^{m-1})
\]
between the group of quasi-isometries and the group of quasi-conformal homeomorphisms of the sphere. 
The topology of $\Isom(\HYP^m)$ coincides with that of $\operatorname{Conf}(S^{m-1})<\Homeo(S^{m-1})$. 

By a result of Tukia~\cite{tukia}, a subgroup $\bar{G}'_1$ of $\operatorname{QConf}(S^{m-1})$ that is
uniformly quasi-conformal and acts cocompactly on triples of points of $S^{m-1}$ 
can be conjugated into $\operatorname{Conf}(S^{m-1})$.

Since the homomorphism $G'\to \QI(X)\to \QI(\HYP^m)$ represents a quasi-action of $G$ on $\HYP^m$, 
the image $\bar{G}'_1\subset \QI(\HYP^m)\cong \operatorname{QConf}(S^{m-1})$ 
is a uniformly quasi-conformal group of homeomorphisms of $S^{m-1}$ due to the Mostow--Morse lemma. 
The homomorphism 
\[
	G'\overto{} \QI(\HYP^m) \hookrightarrow \Homeo(S^{m-1})
\]
is continuous by \cite{Furman-MM}*{Theorem~3.5}.
Since $\Gamma$ acts cocompactly on triples of points in $S^{m-1}$, the same is true for $\bar{G}'_1$. 
By Tukia's result, $\bar{G}'_1$ can be conjugated into the simple Lie group $\operatorname{Conf}(S^{m-1})\cong\Isom(\HYP^m)$.
Thus we obtain a continuous homomorphism
\[
	\rho\colon G'\overto{} \Isom(\HYP^m).
\]
Since $\Isom(\HYP^m)$ is a Lie group, Lemma~\ref{L:NSS} applies, and Theorem~\ref{thm: discreteness in case 3} 
is proved in this case.

\medskip

\subsection*{Case $\Gamma<H$ is uniform and $X_1=\HYP^m_{\bbC}$ with $m\ge 2$.}
This symmetric space is a complex hyperbolic space $\HYP^m_{\bbC}$; 
its boundary sphere $\partial \HYP_{\bbC}^m$ has a natural 
conformal class of a sub-Riemannian Carnot--Carath\'{e}odory structure.
Carnot--Carath\'{e}odory analogues of the theory of quasi-conformal mappings and a result analogous 
to Tukia's theorem imply (cf. \cite{chow}) that $\bar{G}'_1$ can be quasi-conformally conjugated into 
$\operatorname{Conf}(\partial \HYP_{\bbC}^m)=\Isom(\HYP_{\bbC}^m)$. 
This gives us a homomorphism
\[
	\rho\colon G'\overto{} \Isom(\HYP_{\bbC}^m)
\]
that is continuous by \cite{Furman-MM}*{Theorem~3.5}, and takes values
in a Lie group.
So applying Lemma~\ref{L:NSS} we prove Theorem~\ref{thm: discreteness in case 3} in this case.

\subsection*{Case $\Gamma<H$ is uniform and $X_1=\HYP^m_{\bbH}$ or $\HYP^2_{\bbO}$.}
By the result of Pansu~\cite{pansu} in the case of $X_1$ being the quaternionic hyperbolic space $\HYP^m_{\bbH}$
and the Cayley plane $\HYP^2_{\bbO}$ the natural homomorphism $\Isom(X_1)\to \QI(X_1)$ is actually 
an isomorphism. 
The embedding $\Isom(X_1)\hookrightarrow \Homeo(\partial X_1)$ 
is homeomorphic on its image, so the homomorphism 
\[
	\rho\colon G'\overto{} \QI(X_1)\cong \Isom(X_1) 
\]
is continuous (\cite{Furman-MM}*{Theorem~3.5}). 
Thus Lemma~\ref{L:NSS} yields the proof of Theorem~\ref{thm: discreteness in case 3} in this case too.

\subsection*{Case $\Gamma<H$ is uniform and $\operatorname{rk}(X_1)\ge 2$}
For irreducible symmetric space $X_1$ of higher rank,  
the work of Kleiner--Leeb~\cite{Kleiner+Leeb} 
shows that the natural homomorphism $\Isom(X_1)\to \QI(X_1)$ is an isomorphism. 
Thus we obtain a homomorphism 
\[
	\rho\colon G'\overto{} \QI(X_1)\cong \Isom(X_1) 
\]
into a Lie group, and would like to apply Lemma~\ref{L:NSS}.
The proof of Theorem~\ref{thm: discreteness in case 3} would be completed
once we verify continuity of the homomorphism $\rho$.

According to Lemma~\ref{lem: QI of lcsc group} from which we obtained $\rho$, there 
is a constant $C>0$ such that for every bounded subset $B\subset X_1$ there is a neighbourhood of the identity 
$V\subset G'$ such that 
\begin{equation}\label{eq: psi away from id}
	\forall_{g\in V}\forall_{x\in B}~~d_{X_1}(\rho(g)(x),x)<C.
\end{equation}
Let $U\subset \Isom(X_1)$ be an open neighborhood of the identity. 
We have to show that there is open neighborhood of the identity of $G$ that is contained 
in $\rho^{-1}(U)$. 
	
We rely on the following geometric fact about symmetric spaces: 
For every constant $D>0$ and any open neighborhood $W$ of the identity in $\Isom(X_1)$,
there is a bounded set $B\subset X_1$, depending on $D$ and $W$ so that 
\begin{equation}\label{e:geo-cpx}
	\bigl\{ \theta\in \Isom(X_1) \mid \sup_{x\in B}\ d_{X_1}(\theta(x),x)<D\bigr\}\subset W.
\end{equation}
We apply this general fact to the constant $C$ from~\eqref{eq: psi away from id} and 
the identity neighborhood $U$. Let us fix a bounded subset $B=B(C,U)\subset X_1$ such that 
\begin{equation}\label{eq: dist for isometries}
	\bigl\{\theta\in \Isom(X_1) \mid \sup_{x\in B}\ d_{X_1}(\theta(x),x)<C\bigr\}\subset U.
\end{equation}
Applying the statement~\eqref{eq: psi away from id} to this specific subset $B$ provides us with a neighborhood $V=V(C,B)\subset G$ 
of the identity such that 
\begin{equation}\label{eq: dist psi}
	\forall_{g\in V}\forall_{x\in B}~~d_{X_1}(\rho(g)(x),x)<C. 
\end{equation}
Since~\eqref{eq: dist for isometries} and~\eqref{eq: dist psi} 
mean that $V\subset \phi^{-1}(U)$, continuity of $\phi$ follows. 
As mentioned above the now proved Theorem~\ref{thm: discreteness in case 3} also completes the proof of Theorem~\ref{T:Mostow}. 
\end{proof} 

\medskip
\begin{proof}[Proof of Theorem~\ref{T:free-groups}]\hfill{}\\
	Let $\Gamma$ be a finite extension of a finitely generated non-abelian free group,
	and let $\Gamma<G$ be a lattice envelope. 
	Since $\Gamma$ satisfies all the assumptions of Theorem~\ref{T:main} by Theorem~\ref{T:good-groups}, we only need to discuss three possibilities.
	\begin{enumerate}
		\item 
			$\Gamma<G$ is virtually isomorphic to an irreducible lattice in a center-free semisimple
			real Lie group $H$. 
			The first $\ell^2$-Betti number of $\Gamma$ is positive. 
			By Olbrich's work~\cite{olbrich} (see also Chapter~5 about computations of $\ell^2$-invariants of locally symmetric spaces in L\"uck's book~\cite{lueck-book}) 
			the only such lattices with positive first $\ell^2$-Betti numbers are the ones in $\PSL_2(\bbR)$. 
			In this case, $\Gamma$ is a non-uniform lattice, since, for example, $\Gamma$ has virtual cohomological dimension one.
		\item 
			$\Gamma<G$ is virtually isomorphic to an $S$-arithmetic lattices as in Setup~\ref{setup:arith}. 
			In particular, it is virtually isomorphic to a lattice embedding into a product of two non-compact lcsc groups. 
			This can be ruled by the non-vanishing of the first $\ell^2$-Betti number of $\Gamma$ 
			(see the argument in the proof of Theorem~\ref{T: characterisation of groups with positive first L2 Betti}). 
			Alternatively, this situation can ruled out by the fact that S-arithmetic lattices as in~\S\ref{sec:arithmetic_case} are not Gromov hyperbolic. 
	    \item 
			$\Gamma<G$ is virtually isomorphic to a uniform lattice $\Lambda<H$ in a tdlc group~$H$ with trivial amenable radical. 
			Since $\Gamma$ and thus $\Lambda$ are virtually isomorphic to a non-abelian free group, 
	    	we can appeal to the work of Mosher--Sageev--Whyte~\cite{MSW-annals}. 
			Theorem~9 in \emph{loc.~cit.}~states that there is a tree $T$ and a continuous homomorphism 
	    	$H\to \Isom(T)$ with cocompact image and compact kernel. 
			Since $H$ has trivial amenable radical the kernel is trivial. 
			Hence $\Gamma<G$ is virtually isomorphic to a lattice embedding into a closed cocompact subgroup of the automorphism group of a tree.  \qedhere
	\end{enumerate}		
\end{proof}
\medskip
\begin{proof}[Proof of Theorem~\ref{T:surface-groups}]\hfill{}\\
	Let $\Gamma$ be a uniform lattice in $\PSL_2(\bbR)$, and let $\Gamma<G$ be another lattice embedding.
	As a Gromov hyperbolic group, $\Gamma$ satisfies all the assumptions of Theorem~\ref{T:main} according to Theorem~\ref{T:good-groups}. So we only need to discuss three possibilities.
	\begin{enumerate}
	\item $\Gamma<G$ is virtually isomorphic to an irreducible lattice is a semisimple
	real Lie group $H$. As in the proof above, by positivity of the first $\ell^2$-Betti number of $\Gamma$ one is reduced to the case of $H\simeq\PSL_2(\bbR)$.
	In this case, $\Gamma$ has to be cocompact, but the two embeddings of $\Gamma\to \PSL_2(\bbR)$ need not be conjugate. 
	\item $\Gamma<G$ is virtually isomorphic to an $S$-arithmetic lattices as in Setup~\ref{setup:arith}. This case is ruled out for the same reason as case (2) in the proof of Theorem~\ref{T:free-groups}. 
	\item $\Gamma<G$ is virtually isomorphic to a uniform lattice in a tdlc group. This 
	is covered by~\cite{Furman-MM}*{Theorem~C}: the only possibility is a trivial lattice embedding.\qedhere
	\end{enumerate}
\end{proof}

\medskip 
\begin{proof}[Proof of Theorem~\ref{T:GromovThurston}]\hfill{}\\
	Let $M$ be a closed Riemannian manifold of dimension at least~$5$ with sectional curvatures ranging in 
	\[ \bigl[-\bigl(1+\frac{1}{n-1}\bigr)^2, -1\bigr].\]
		
	Let $\Gamma=\pi_1(M)$ be the fundamental group, and let $\Gamma<G$ be a lattice envelope of $\Gamma$. 
	Since $\Gamma$ is finitely generated, $G$ is compactly generated by Lemma~\ref{L:fg-cg}. 	
	The group $\Gamma$ is Gromov hyperbolic, so it has the properties 
	\CAF, \NBC, and \IRR~according to Theorem~\ref{T:good-groups}. 	
	As a fundamental group of a closed aspherical manifold, $\Gamma$ is torsionfree. In particular, $\Gamma$ also has property~\BT. 
	By Theorem~\ref{T:main} there are the following three possibilities for the lattice embedding $\Gamma<G$ up to virtual isomorphism. Since $\Gamma$ is torsionfree, we may assume upon replacing $\Gamma$ by a finite index subgroup, $G$ by a finite index subgroup and a quotient by a compact normal subgroup, and $M$ by a finite cover that we still have $\Gamma=\pi_1(M)$ and $\Gamma<G$ is isomorphic (not just virtually) to one of the following cases. 
	\begin{enumerate}
	\item $\Gamma<G$ is an (irreducible) lattice embedding into a semisimple
	real Lie group $G$.
	\item $\Gamma<G$ is an S-arithmetic lattice in the sense of Setup~\ref{setup:arith}. 
	\item $\Gamma<G$ is a uniform lattice and $G$ is a tdlc group with trivial amenable radical. 
	\end{enumerate}

    We need to show that the only possibility is case (3) with $G$ being discrete unless 
    $M$ is homeomorphic to a hyperbolic manifold. 
    
    Since $\Gamma$ is Gromov hyperbolic we can rule out case (2). By the same reason, case (1) is only possible 
    if $G$ has real rank~$1$ and $\Gamma<G$ is uniform; note here that non-uniform lattices contain a free abelian subgroup of rank~$2$. Let us analyze the situation where $\Gamma<G$ is a 
    uniform lattice in a simple Lie group of real rank~$1$. Let $X=G/K$ be the associated symmetric space which is thus real, complex, quaternionic or Cayley hyperbolic. 
	As aspherical spaces with the same fundamental group, the locally symmetric space $\Gamma\bs X$ and $M$ are 
	homotopy equivalent. 
	
    If $X$ is real hyperbolic, then $M$ is homeomorphic to the closed hyperbolic 
    manifold $\Gamma\bs X$ by the following striking result of Farrell--Jones~\cite{farrell+jones-rigidity} in their work on the Borel conjecture. 
	So this possibility is ruled out by assumption.  
    \begin{theorem}[Topological rigidity~\cite{farrell+jones-rigidity}]
    Let $Y$ be a closed non-positively curved manifold of dimension $\ne 3,4$. If a closed manifold $Z$ is homotopy equivalent to $Y$, then $Y$ and $Z$ are homeomorphic. 
    \end{theorem}

    The possibilities that $X$ is complex, quaternionic or Cayley hyperbolic are ruled out by applying 
    the following result by Mok, Siu, and Yeung~\cite{MSY} to $Y=\Gamma\bs X$ and $Z=M$ and the pinching assumption for $M$. 
	          
    \begin{theorem}[Geometric rigidity~\cite{MSY}*{Theorem~1}]
    Let $Y$ and $Z$ be homotopy equivalent closed Riemannian manifolds. Assume that $Z$ is negatively curved and $Y$ is complex, quaternionic or Cayley hyperbolic. Then $Y$ and $Z$ are isometric up to scaling. In particular, the sectional curvatures of $Z$ cannot range in $[-a, -1]$ 
    for $a<4$.     
    \end{theorem}
    Summing up, case (2) never occurs, and case (1) can only occur if $M$ is homeomorphic to hyperbolic manifold. 
    
    Finally, let us consider case (3) in which we have to show that the tdlc group $G$ is discrete upon dividing out a compact normal subgroup. 
	Consider the natural homomorphism 
	\[
		\phi\colon \Gamma\to\Isom(\tilde M)\to  \QI(\tilde M)\to \Homeo(\partial\tilde M)=\Homeo(S^{n-1}).
	\]
	Of course, $\phi$ coincides with the natural action $\Gamma\to \Homeo(\partial\Gamma)$ of 
	$\Gamma$ on its Gromov boundary. We can refer to Lemma~\ref{lem: QI of lcsc group} or, in this situation, to~\cite{Furman-MM}*{Theorem~3.5} to conclude that 
	the map $\phi$ extends to a 
	homomorphism \[\bar\phi\colon G\to\Homeo(S^{n-1}).\] 
	In the latter reference it is also shown that $\bar\phi$ is continuous, the kernel $K:=\ker(\bar\phi)$ is compact, and the image of $\bar\phi$ is locally compact in the subspace topology. By the open mapping theorem~\cite{bourbaki}*{IX. \S 5.3} $G$ is topologically isomorphic to $\im(\bar\phi)$. Since $G$ has trivial amenable radical, $K$ is trivial. 
	
	Since the tdlc group $G$ acts continuously and faithfully on the sphere $S^{n-1}$, a positive 
	solution of the Hilbert-Smith conjecture would imply that $G$ is discrete and thus finish the proof. The general Hilbert-Smith conjecture remains open but in our situation we can appeal to the work of Mj~\cite{Mj} that contains the following result. 
	\begin{theorem}[Hilbert-Smith conjecture for boundary actions~\cite{Mj}*{Corollary~2.2.5}]\label{thm: hilbert smith}
	Let $\Lambda$ be a Gromov hyperbolic Poincar\'e duality group. 
	
	Let $Q<\Homeo(\partial\Lambda)$ be a subgroup that is finite dimensional and locally compact in the subspace topology and lies in the image of the natural homomorphism $\QI(\Lambda)\to\Homeo(\partial\Lambda)$. Let 
	$h$ be the Hausdorff dimension of $\partial\Lambda$ with respect to the visual metric, and let $t$ be the topological dimension of $\partial\Lambda$. We assume that $h<t+2$. 
	
	Then $Q$ is a Lie group. 		
	\end{theorem}
    We apply this theorem to $Q=\im(\bar\phi)\cong G$ and $\Lambda=\Gamma$. The topological dimension of $\partial \Gamma=S^{n-1}$ is $t=n-1$. It remains to verify that we have $h<t+2=n+1$. The Hausdorff dimension $h$ equals the volume entropy $h_{vol}$ 
    of $\tilde M$ (see~\cite{yue}*{Theorem~C}), that is, 
    \[ h= h_{vol}:=\limsup_{R\to\infty} \frac{1}{R}\log\vol\bigl(B_{\tilde M}(x, R)\bigr).\]
    As a direct consequence of the Bishop-Gunther comparison theorem~\cite{gallot+hulin+lafontaine-book}*{Theorem~3.101 on p.~169} we obtain that if the $n$-dimensional manifold $\tilde M$  
    has sectional curvature in $[-b^2,-a^2]$, then 
    \[  (n-1)a \le h_{vol}\le (n-1)b.\]
    In the present situation, $b=1+1/(n-1)$, so we obtain that 
    \[ h=h_{vol}\le (n-1)(1+\frac{1}{n-1})=n<n+1=t+2.\]
    So the assumptions in Theorem~\ref{thm: hilbert smith} are satisfied and we conclude that $G$ is at the same time a Lie group and a tdlc group, thus discrete. 
\end{proof}

\medskip 
\begin{proof}[Proof of Theorem~\ref{T: characterisation of groups with positive first L2 Betti}]\hfill{}\\
	Let $\Gamma$ be group with positive first $\ell^2$-Betti number $\beta^{(2)}_1(\Gamma)>0$ and property~$\BT$.
	Assume that $\Gamma$ admits a non-uniform compactly generated lattice envelope $G$.
	Since, by Theorem~\ref{T:good-groups}, $\Gamma$ satisfies all the conditions of Theorem~\ref{T:main} and~\BT, $\Gamma<G$ is virtually isomorphic to some lattice embedding $\Gamma'<G'$ that is one of the first two types in Theorem~\ref{T:main}; the totally disconnected case is ruled out by~\BT. Non-vanishing of the first $\ell^2$-Betti number 
	is preserved under virtual isomorphism (this easily deduced from their basic properties), even under quasi-isometry~\cite{sauer+schroedl}*{Theorem~1}.  
	So $\beta^{(2)}_1(\Gamma')>0$. Hence $\Gamma'$ cannot be a lattice in a product of two non-compact lcsc groups. That is because $\beta^{(2)}_1(\Gamma')>0$ implies that the first $\ell^2$-Betti number of any lattice envelope of $\Gamma'$ is positive~\cite{kyed+petersen+vaes}*{Theorem~B}, and the first $\ell^2$-Betti number of a product of two non-compact unimodular lcsc groups is zero~\cite{petersen}*{Theorem~6.5 on p.~61}. 
	This excludes that $\Gamma'<G'$ is an S-arithmetic lattice as in Setup~\ref{setup:arith}, so type (2) is ruled out. 
	Hence $\Gamma'<G'$ is a lattice in a connected center-free semisimple Lie group without compact 
	factors. By Olbrich's work~\cite{olbrich} (see also Chapter~5 about computations of $\ell^2$-invariants of locally symmetric spaces in L\"uck's book~\cite{lueck-book}) 
	the only such lattices with positive first $\ell^2$-Betti numbers are the ones in 
	$\PSL_2(\bbR)$. Since $\Gamma'<G'$ is also non-uniform, the group $\Gamma'$, hence $\Gamma$, is virtually isomorphic to a free group. 
\end{proof}

\end{document}